\documentclass[a4paper]{article}

\usepackage{amsfonts}
\usepackage{amssymb,amsthm}
\usepackage{graphicx}
\usepackage{amsmath}
\usepackage[a4paper]{geometry}
\usepackage{comment}
\usepackage{harvard}

\swapnumbers

\newtheorem{theorem}{Theorem}

\newtheorem{assumption}[theorem]{Assumption}

\newtheorem{corollary}[theorem]{Corollary}

\newtheorem{definition}[theorem]{Definition}

\newtheorem{lemma}[theorem]{Lemma}

\newtheorem{proposition}[theorem]{Proposition}

\numberwithin{equation}{section}

\swapnumbers

\DeclareMathOperator{\E}{{\mathbb E}}

\DeclareMathOperator{\R}{{\mathbb R}}

\DeclareMathOperator{\N}{{\mathbb N}}

\DeclareMathOperator{\supp}{supp}

 \DeclareMathOperator{\sgn}{sgn}

 \DeclareMathOperator{\Id}{Id}

 \DeclareMathOperator{\Poiss}{Poiss}

\providecommand{\eps}{\varepsilon}
\renewcommand{\phi}{\varphi}
\renewcommand{\theta}{\vartheta}
\renewcommand{\subset}{\subseteq}

\renewcommand{\cdot}{{\scriptstyle \bullet} }
\providecommand{\abs}[1]{\lvert #1 \rvert}
\providecommand{\norm}[1]{\lVert #1 \rVert}

\providecommand{\babs}[1]{{\Bigl\lvert #1 \Bigr\rvert}}
\providecommand{\scapro}[2]{\langle #1,#2 \rangle}

\renewcommand{\Re}{\operatorname{Re}}

\renewcommand{\le}{\leqslant}
\renewcommand{\ge}{\geqslant}

\newcommand{\cit}[1]{\citeasnoun{#1}}

\begin{document}

\title{A Donsker Theorem for L\'evy Measures}
\author{Richard Nickl ~~~~~~~~ Markus Rei\ss  \\
\\
\textit{University of Cambridge \footnote{Statistical Laboratory, Department of Pure Mathematics and Mathematical Statistics, University of Cambridge, CB30WB, Cambridge, UK. Email: r.nickl@statslab.cam.ac.uk}} ~ and  \textit{Humboldt-Universit\"at zu Berlin \footnote{Institut f\"ur Mathematik, Humboldt-Universit\"at zu Berlin, Unter den Linden 6, 10099 Berlin, Germany. Email: mreiss@math.hu-berlin.de }}}
\date{First version: January 3, 2012, this version: \today}

\maketitle

\begin{abstract}
Given $n$ equidistant realisations of a L\'evy process $(L_t,\,t\ge 0)$, a
natural estimator $\hat N_n$ for the distribution function $N$ of the L\'evy
measure is constructed. Under a polynomial decay restriction on the
characteristic function $\phi$, a Donsker-type theorem is proved, that is, a
functional central limit theorem for the process $\sqrt n (\hat N_n -N)$ in the
space of bounded functions away from zero. The limit distribution is a
generalised Brownian bridge process with bounded and continuous sample paths
whose covariance structure depends on the Fourier-integral operator ${\cal
F}^{-1}[1/\phi(-\cdot)]$. The class of L\'evy processes covered includes several relevant
examples such as compound Poisson, Gamma and self-decomposable processes. Main ideas in the proof include establishing pseudo-locality of
the Fourier-integral operator and recent techniques from smoothed empirical
processes. 

\medskip

\noindent\textit{MSC 2010 subject classification}: Primary: 46N30;\ Secondary:
60F05.

\noindent\textit{Key words and phrases: uniform central limit theorem,
nonlinear inverse problem, smoothed empirical processes, pseudo-differential
operators, jump measure.}
\end{abstract}


\section{Introduction}

A classical result of probability theory is Donsker's central limit theorem for empirical distribution functions: If $X_1, \dots, X_n$ are i.i.d.~random variables with distribution function $F(t)=P((-\infty, t]),\, t \in \mathbb R$, and if $F_n(t)=P_n((-\infty, t])$ where $P_n = \frac{1}{n} \sum_{k=1}^n \delta_{X_k}$ is the empirical measure, then $\sqrt n (F_n-F)$ converges in law in the Banach space of bounded functions on $\mathbb R$, to a $P$-Brownian bridge. The result in itself and its many extensions have been at the heart of much of our understanding of modern statistics, see the monographs \cit{D99}, \cit{VW96} for a comprehensive account of the foundations of this theory.

The purpose of this article is to investigate a conceptually closely related
problem: at equidistant time steps $t_k=k\Delta$, $k=0,1,\ldots,n$, one
observes a trajectory of a L\'{e}vy process  with corresponding L\'{e}vy (or
jump) measure $\nu$, and wishes to estimate the distribution function $N$ of
$\nu$. Since we do not assume that the time distance $\Delta$ varies (in
particular, no high-frequency regime), we equivalently observe a sample from an
infinitely divisible distribution given by the i.i.d.~increments of the
process. Since $\nu$ is only a finite measure away from zero the natural target
of estimation is $N(t) = \nu((-\infty, t])$ for $t<0$ and
$N(t)=\nu([t,\infty))$ for $t>0$. By analogy to the classical case of
estimating $F$, one aims for an estimator $\hat N$ such that $\sqrt n(\hat N -
N)$ satisfies a limit theorem in the space of functions bounded on $\mathbb R
\setminus (-\zeta, \zeta), \zeta>0$. Statistical minimax theory reveals that
the problem of estimating $N$ is intrinsically more difficult than the one of
estimating $F$ -- it is a nonlinear inverse problem in the terminology of
nonparametric statistics. We discuss this point in more detail below, but note
that it implies that a rate of convergence $1/\sqrt n$ for $\hat N(t)-N(t)$,
even only at a single point $t$, cannot be achieved (by any estimator $\hat N$)
without certain qualitative assumptions on the L\'{e}vy process. Particularly,
the process cannot contain a nonzero Gaussian component. On the other hand, and
perhaps surprisingly, we show in the present article that for a large and
relevant class of L\'evy processes a Donsker theorem can be proved.

\medskip

Similar to Donsker's classical theorem our results have interesting
consequences for statistical inference, such as the construction of confidence
bands and goodness of fit tests. While we do not address these issues
explicitly here and concentrate on spelling out the mathematical ideas, it is
nevertheless instructive to discuss some related literature on statistical
inference on the L\'evy triplet from discrete observations. The basic principle
for understanding the nonlinearity in this setting is already inherent in the
problem of {\it decompounding} a compound Poisson process, which has been
studied in queuing theory and insurance mathematics. In this case the L\'evy
measure $\nu$ is a finite measure and by explicit inversion in the convolution
algebra \cit{BG03} prove a central limit theorem with rate $1/\sqrt{n}$ for a
plug-in estimator of $N$ in an exponentially weighted supremum norm, assuming
that the intensity of the process is known.

For general L\'evy triplets the estimation problem is generally ill-posed in
the sense of inverse problems. In fact, the linearized problem is of
deconvolution-type where the part of the error distribution is taken over by
the observation law itself. This phenomenon, which could be coined {\it
auto-deconvolution}, was first studied by \cit{BR06}. For the general problem
of estimating functionals of the L\'evy measure the results by \cit{NR09} show
in particular that a functional can be estimated at parametric rate
$1/\sqrt{n}$ provided its smoothness outweighs the ill-posedness induced by the
decay of the characteristic function. Comparing to \cit{NR09} we are thus
interested in the low regularity functional $f \mapsto \int_{-\infty}^t f$ (not
covered by their results), and in exact limiting distributions.  Instead of
making inference on the distribution function, one may also be interested in
the associated nonparametric estimation problem for a Lebesgue density of the
L\'evy measure, where the rate $1/\sqrt n$ can never be attained. This problem
was studied in \cit{G09} for L\'evy processes with finite jump activity and a
Gaussian part, \cit{CG10} for a model selection procedure in the finite
variation case, or \cit{T11} for self-decomposable processes. Generalisations
for observations of more general jump processes like L\'evy-Ornstein-Uhlenbeck
processes or affine processes are considered by \cit{JMV05} and \cit{B11}.

\medskip

The proof of our main result contains certain subtleties that we wish to
briefly discuss here: In the classical Donsker case one proves that the
empirical process $\sqrt n (P_n-P)$ is tight in the space of bounded mappings
acting on $\{1_{(-\infty, t]}: t \in \mathbb R\}$. The ill-posedness of the
L\'evy-problem can be roughly understood, after linearisation, as requiring to
show that the empirical process $\sqrt n(P_n-P)$ is tight in the space of
bounded mappings acting on the class
\begin{equation}\label{EqIntro}
\mathcal G_{\phi}=\{{\cal F}^{-1}[1/\phi(-\cdot)]\ast 1_{(-\infty, t]} : |t| \ge \zeta)\},
\end{equation}
where $\zeta>0$ is arbitrary, ${\cal F}$ is the Fourier transform and where $\phi = \mathcal FP$ is the characteristic function of the increments of the L\'{e}vy process. In fact, the situation is more complicated than that, but the above simplification highlights the main problem. Convolution with ${\cal F}^{-1}[1/\phi]$ is just a way of writing deconvolution with $P={\cal F}^{-1}[\phi]$, which is mathematically understood as the action of a pseudo-differential operator, and the class $\mathcal G_{\phi}$ can be shown \textit{not} to be $P$-Donsker (arguing as in Theorem 7 in \cit{N06}, for instance), unless in very specific situations (effectively in the compound Poisson case discussed above). In other words, the empirical process is not tight when indexed by these functions.

A starting point of our analysis is that for certain L\'{e}vy processes a generalised $P$-Brownian bridge $\mathbb G^\phi$
with bounded sample paths can be defined on $\mathcal G_\phi$, uniformly
continuous for the intrinsic covariance metric of $\mathbb G^\phi$, see Theorem \ref{scgauss}. Roughly speaking this means that a tight limit process exists, and that a limit theorem at rate $1/\sqrt n$ may hold if one
replaces the empirical process by a smoothed one. This hope
is nourished by the phenomenon -- first observed, in a general empirical
process setting unrelated to the present situation, by \cit{RW00}, and recently
developed further in several directions by \cit{GN08} -- that smoothed
empirical processes may converge in situations where the unsmoothed process
does not. The results in \cit{GN08} apply to unbounded classes, so in particular to $\mathcal G_\varphi$, and this idea in combination with a
thorough analysis of the pseudo-differential operator ${\cal
F}^{-1}[1/\phi(-\cdot)]$ are at the heart of our proofs.

The paper is organised as follows: Section \ref{setm} contains the exact conditions on the model, the construction
of the estimator and the main result. In Section 3 the model assumptions, some
important examples and potential extensions are discussed. Finally, the
complete proof of the Donsker-type result is given in Section 4, divided into
the finite-dimensional central limit theorem and the uniform tightness result.

\section{The Setting and Main Result} \label{setm}

We observe a real-valued L\'evy process $(L_t,\,t\ge 0)$ at equidistant time points $t_k=k\Delta$, $k=0,1,\ldots,n$, for $\Delta>0$ fixed. It will be seen to be natural (Section \ref{discu}) to restrict to L\'evy processes of (locally) finite variation. In this case the characteristic function of the increments $X_k:=L_{t_k}-L_{t_{k-1}}$  is given by
\[ \phi(u)=\E[\exp(iuL_\Delta)]=e^{\Delta\psi(u)}\text{ where } \psi(u)=i\gamma u + \int_{\R\setminus\{0\}} (e^{iux}-1)\,\nu(dx)
\]
with drift parameter $\gamma\in\R$ and L\'evy (or jump) measure $\nu$ satisfying $\int_{\R} (\abs{x}\wedge 1)\,\nu(dx)<\infty$ (due to finite variation). The increments $X_1,\ldots,X_n$ are i.i.d. and we write $P$ for the law of $X_k$ and $p$ for its density (if it exists) as well as $P_n=\frac 1n\sum_{k=1}^n\delta_{X_k}$ and $\phi_n(u)={\cal F}P_n(u)=\int e^{iux}dP_n(x)$ for the empirical measure and empirical characteristic function, respectively. Throughout $\cal F$ denotes the Fourier (-Plancherel) transform acting on finite measures, on the space $L^1(\mathbb R)$ of integrable or on the space $L^2(\mathbb R)$ of square-integrable functions on $\mathbb R$, see e.g. \cit{K76} for the standard Fourier techniques that we shall employ.

If $\nu$ has a finite first moment, then the weighted L\'evy measure $x\nu(dx)$
can be identified directly from the law of $X_k$ in the Fourier domain:
\begin{equation}\label{EqIdent} \frac{1}{i\Delta}\frac{\phi'(u)}{\phi(u)}=-i\psi'(u)=\gamma+\int e^{iux}x\nu(dx)=\gamma+{\cal F}[x\nu](u).
\end{equation}
Our goal is to estimate the cumulative distribution function of $\nu$,
\begin{equation}\label{EqN}
N(t):=\begin{cases}\nu((-\infty,t]),& t<0, \\ \nu([t,\infty)), & t>0,\end{cases}
\end{equation}
from the sample $X_1, \dots, X_n$.  Note that in general $N(t)$ tends to infinity for $t\to 0$. If we denote by ${\cal F}^{-1}$ the inverse Fourier transform, then the relation \eqref{EqIdent} suggests a natural empirical estimate of $N(t)$ (we shall see below that $\gamma$ can be neglected),
\begin{equation}\label{EqNhat}
\hat N_n(t):=\int_{\R}g_t(x){\cal F}^{-1}\left[\frac{1}{i\Delta}\frac{\phi_n'}{\phi_n} {\cal F}K_h \right](x)\,dx\text{ with } g_t(x):=\begin{cases} x^{-1}{\bf 1}_{(-\infty,t]}(x),&\;t<0,\\x^{-1}{\bf 1}_{[t,\infty)}(x),&\;t>0,\end{cases}
\end{equation}
where $K$ is a band-limited kernel function  and $K_h(x):=h^{-1}K(x/h)$. In the sequel the kernel will be required to satisfy
\begin{equation}\label{EqKProp}
\int K=1, \quad \supp({\cal F}K)\subset[-1,1]\text{ and } \abs{K(x)}+\abs{K'(x)}\lesssim (1+\abs{x})^{-\beta} \text{ for some $\beta>2$.}
\end{equation}
Throughout, we shall write $A_p\lesssim B_p$ if $A_p\le CB_p$ holds with a uniform constant $C$ in the parameter $p$ as well as $A_p\thicksim B_p$ if $A_p\lesssim B_p$ and $B_p\lesssim A_p$.

The smooth spectral cutoff induced by multiplication with ${\cal F}K_h$ is desirable for various reasons; in particular, it will imply that $\hat N_n$ is well-defined with probability tending to one. By Plancherel's formula, we have the alternative representation
\[ \hat N_n(t):=\frac{1}{2\pi i\Delta}\int_{\R}{\cal F}g_t(-u)\frac{\phi_n'(u)}{\phi_n(u)} {\cal F}K_h(u)\,du.
\]
Heuristically, for $h_n \to 0$ we expect consistency $\hat N_n(t) \to N(t)$ in probability, $t\not=0$, because as $h_n \to 0$ we have $K_{h_n}\to \delta_0$ (the Dirac measure in zero)
and thus ${\cal F}K_{h_n}(u)\to 1$ which may be combined with the law of large numbers for both $\phi_n$ and $\phi_n'$. For this argument to work it is important to note that the drift $\gamma$ induces a point measure in zero for ${\cal F}^{-1}[\phi'/\phi]$ which is outside the support of $g_t$, compare Section \ref{SecDrift} below. For our precise results we shall need the following conditions on the data-generating L\'evy process. Throughout the paper we often write $\phi^{-1}$ for $1/\phi$.

\begin{assumption} \label{gen}
We require for some $\eps>0$:
\begin{enumerate}
\item $\int \max(\abs{x},\abs{x}^{2+\eps})\,\nu(dx)<\infty$;
\item $x\nu$ has a bounded Lebesgue density and $\abs{{\cal F}[x\nu](u)}\lesssim (1+\abs{u})^{-1}$;
\item $(1+\abs{u})^{-1+\eps}\phi^{-1}(u)\in L^2(\R)$.
\end{enumerate}
\end{assumption}

Assumption \ref{gen}(a) imposes finite variation, ensuring the identification identity \eqref{EqIdent}, as well as finite $(2+\eps)$-moments of $\nu$ and $P$, since by Thm. 25.3 in \cit{Sa99}
\begin{equation}\label{EqMoments}
\int_{\R} \abs{x}^{2+\eps}\nu(dx)<\infty \iff \int_{\R} \abs{x}^{2+\eps}P(dx)<\infty.
\end{equation}
As $\hat N$ is based on $\phi_n'(u)$, and since a central limit theorem is desired, it is natural to require a finite second moment of $X_k$. The additional $\eps$ in the power will allow to apply the Lyapounov criterion in the CLT for triangular schemes and to obtain uniform in $u$ stochastic bounds for $\phi_n'(u)-\phi'(u)$ over increasing intervals. Assumptions \ref{gen}(b,c) are discussed in more detail after the following theorem, which is the main result of this article.

For $\zeta>0$, let $\ell^\infty((-\infty, -\zeta] \cup [\zeta, \infty))$ be the space of bounded real-valued functions on $(-\infty, -\zeta] \cup [\zeta, \infty)$ equipped with the supremum norm. Convergence in law in this space, denoted by $\to^\mathcal L$, is defined as in \cit{D99}, p.94.

\begin{theorem}\label{Thm1}
Suppose that Assumption \ref{gen} is satisfied, $\zeta>0$ and $h_n\thicksim n^{-1/2}(\log n)^{-\rho}$ for some $\rho>1$. Then as $n \to \infty$
$$\sqrt n (\hat N_n - N) \to^\mathcal L \mathbb G^\varphi~~\text{in}~\ell^\infty((-\infty, -\zeta] \cup [\zeta, \infty)),$$
where $\mathbb G^\varphi$ is a centered Gaussian Borel random variable in $\ell^\infty((-\infty, -\zeta] \cup [\zeta, \infty))$ with covariance structure given by
$$ \Sigma_{t,s}=\frac{1}{\Delta^2} \int_\mathbb R \Big({\cal F}^{-1}\left[\frac{1}{\phi(-\cdot)}\right] \ast (xg_t(x))\Big) \times \Big({\cal F}^{-1} \left[\frac{1}{\phi(-\cdot)}\right] \ast (xg_s(x))\Big)\, P(dx)$$
and where $g_t$ is given in \eqref{EqNhat}.
\end{theorem}

In view of $xg_t(x)={\bf 1}_{(-\infty,t]}(x)$ for $t<0$ and symmetrically for $t>0$, the representation of the covariance in the theorem above is intuitively appealing when compared to the classical Donsker theorem. Its rigorous interpretation, however, needs some care, as it is not quite clear how the pseudo-differential operator ${\cal F}^{-1}[\phi^{-1}(-\cdot)]$ acts on the indicator function $x g_t(x)$.  One rigorous representation that follows from our proofs uses
$${\cal F}^{-1}\left[\phi^{-1}(-\cdot)\right] \ast 1_{(-\infty, t]} = {\cal F}^{-1}\left[(1+iu)^{-1}\phi^{-1}(-u)\right] \ast (1_{(-\infty, t]}+\delta_t)
$$
together with the fact that ${\cal F}^{-1}[(1+iu)^{-1}\phi^{-1}(-u)]$ can be shown to be contained in $L^1(\R) \cap L^2(\R)$ under Assumption \ref{gen} (using lifting properties of Besov spaces), so that the right-hand side of the last display is defined almost everywhere.

Another more explicit representation, which also implies that $\Sigma_{t,t}<\infty$, is the following: Note that formally
\[
\int_\mathbb R {\cal F}^{-1}\left[\frac{1}{\phi(-\cdot)}\right] \ast (xg_t(x))dP(x) =\frac{1}{2\pi}\int_{\R}({\cal F}[xg_t](-\cdot))(u)\phi^{-1}(u)\phi(u) \,du
=(xg_t)(0)=0,
\]
which explains why the covariance in Theorem \ref{Thm1} is centered for $t \neq 0$. Moreover, ${\cal F}[xg_t]=i^{-1}({\cal F}[g_t])'$ and integration by parts gives rise to the formally equivalent representation
\begin{equation} \label{gausscov}
\Sigma_{t,s} = (i\Delta)^{-2} \int_\mathbb R h_t(x) h_s(x) P(dx)
\end{equation}
where
\begin{equation*}
h_t(x) = {\cal F}^{-1}[\phi^{-1}(-u){\cal F}g_t(u)](x)ix+{\cal F}^{-1}[(\phi^{-1})'(-u){\cal F}g_t(u)](x),
\end{equation*}
and where we note that $i^{-1}h_t$ is real-valued. This expression for $h_t$ is the one we shall employ in our proofs, as it can be shown to be rigorously defined in $L^2(P)$ under the maintained assumptions, see (\ref{htrig}) below for more details.

Moreover the last representation immediately suggests consistent estimators of $\Sigma_{t,s}$ based on the empirical characteristic function $\phi_n$ and the empirical measure $P_n$, useful when one is interested in the Gaussian limiting distribution for inference purposes on $N$.

\section{Discussion} \label{discu}

\subsection{The regularity conditions}

We remark first that the results in \cit{NR09} imply that we can attain a $1/\sqrt{n}$-rate for estimation only if the characteristic function decays at most with a low polynomial order. This restricts the classes of L\'evy processes automatically to the (locally) finite variation case (e.g. proof of Prop. 28.3 in \cit{Sa99}), and moreover excludes all L\'evy processes with a nonzero Gaussian component.

Let us next discuss Assumption \ref{gen}(c) which describes the lower bound we
need on the ill-posedness of the estimation problem. It holds for all compound
Poisson processes, in which case $\abs{\phi^{-1}(u)}$ is bounded, but also for
Gamma processes with $\alpha\in(0,1/(2\Delta))$ and  for pure-jump
self-decomposable processes with not too high jump activity at zero, see
Proposition \ref{PropEx} below. Recall (e.g. \cit{Sa99}, Section 15) that
self-decomposable distributions describe the limit laws of suitably rescaled
sums of independent random variables as well as the stationary distributions of
L\'evy-Ornstein-Uhlenbeck processes, and thus give rise to a rich nonparametric
class of L\'evy measures. More generally, if $\E[e^{iuL_1}]$ decays polynomially, then there exists
a $\Delta_0>0$ such that for all $\Delta<\Delta_0$ the corresponding
characteristic function $\phi(u)=\E[e^{iuL_\Delta}]$ satisfies
$\abs{\phi^{-1}(u)}\lesssim (1+\abs{u})^\alpha$ for $\alpha<1/2$, so
Assumption \ref{gen}(c) holds for any polynomially decaying $\phi$ if the sampling frequency is large (i.e., $\Delta$ small) enough. Abstractly, Assumption \ref{gen}(c) means that the
pseudo-differential operator ${\cal F}^{-1}[\phi^{-1}]$ of deconvolution is an
element of the $L^2$-Sobolev space $H^{-1+\eps}(\R)$ of negative order
$\eps-1$. In the simpler problem of statistical deconvolution an analogous
restriction for the characteristic function of the error variables is
necessary, even if one is only interested in rates of convergence of an
estimator, and the situation is similar here: The lower bound techniques from
Theorem 4.4 of \cit{NR09} or Theorem 1 of \cit{LN11} can be adapted to the
present situation to imply, for instance, that for Gamma processes with $\alpha
> 1/(2\Delta)$ the 'parametric' rate $1/\sqrt n$ cannot be achieved by any
estimator in the L\'evy estimation problem considered here, so that Assumption
\ref{gen}(c) is in this sense sharp for Theorem \ref{Thm1}.

The smoothness condition on $x\nu$ in Assumption \ref{gen}(b) is not very restrictive: it is satisfied whenever the weighted L\'evy measure $x\nu$ has a density whose weak derivative is a finite measure (noting $x\nu\in L^1(\R)$ by Assumption \ref{gen}(a)). As simple examples, any compound Poisson process with a jump density of bounded variation and a finite first moment satisfies this condition, as does any Gamma process. More generally, most self-decomposable processes  satisfy this condition, see Proposition \ref{PropEx} below.

The key role of Assumption \ref{gen}(b) is not to enforce smoothness of $\nu$, but to ensure {\it pseudo-locality} of the deconvolution operator
${\cal F}^{-1}[\phi^{-1}]$ in the sense that the location of singularities like
the jump in the indicator ${\bf 1}_{(-\infty,t]}$ remains unchanged under
deconvolution. A similar situation arises in standard deconvolution problems, see the recent paper \cit{SHMD12}. In the spirit of the theory of pseudo-differential operators this
is established by differentiating in the spectral domain, see \eqref{EqPsiId}
below for details,
\[{\cal F}^{-1}[\phi^{-1}(-u)]=\frac{1}{i\cdot}{\cal F}^{-1}[(\phi^{-1}(-u))']\]
under the condition that $(\phi^{-1})'=\Delta\psi'\phi^{-1}\in L^2(\R)$.
Neglecting the drift, $\psi'$ is ${\cal F}[ix\nu]$ and Assumptions
\ref{gen}(b), \ref{gen}(c) together ensure $(\phi^{-1})'\in L^2(\R)$, see Lemma
\ref{LemmaLevyProp} below. As discussed later, the example of a superposition
of a Gamma and Poisson process provides a simple concrete situation where a
violation of this condition renders the asymptotic variance in Theorem
\ref{Thm1} infinite.

There is another interesting interaction between Assumptions \ref{gen}(b) and
\ref{gen}(c). A decay rate $\abs{u}^{-1}$ for ${\cal F}[x\nu](u)$ is the
maximal possible smoothness requirement under \ref{gen}(c); otherwise
$\abs{\Re(\psi'(u))}\le\abs{{\cal F}[x\nu](u)}=o(\abs{u}^{-1})$ would imply
$\abs{\phi(u)}=\exp(\Re(\Delta\psi(u)))=\exp(o(\log(u)))$ for $\abs{u}\to\infty$,
excluding polynomial decay of the characteristic function $\phi$.

\subsection{Examples}

We now discuss a few examples in more detail.

\begin{description}
\item[Compound Poisson Processes.]
The compound Poisson case where $\nu$ is a finite measure is covered in
Theorem \ref{Thm1}. Note that due to the presence of a point mass at zero
in $P$ the characteristic function satisfies $\inf_u\abs{\phi(u)}\ge
\exp(-2\nu(\R))>0$ ($\Delta=1$). Therefore Assumption \ref{gen}(c) is trivially
satisfied. Assumption \ref{gen}(b) requires that the law of the jump sizes
has a density $\nu$ such that $x\nu(x)$ is bounded and has the respective
decay property in the Fourier domain. Assumption \ref{gen}(a) just
postulates $(2+\eps)$ finite moments of the jump law. Compared to
\cit{BG03} we thus obtain directly a uniform central limit without
weighting, exponential moments and, perhaps more importantly, without prior knowledge of the intensity, yet our result
holds only away from the origin and under Assumption \ref{gen}(b).

Stronger results can be obtained by adapting our method to this specific
case because the distribution function $N$ of $\nu$ is defined classically
for all $t\in\R$ and Assumption \ref{gen}(b) is not required to ensure
pseudo-locality of deconvolution. In fact, deconvolution reduces to
convolution with a signed measure because of ($\bar\nu^{\ast k}$ denotes $k$-fold convolution)
\[ {\cal F}^{-1}[\phi^{-1}(-\cdot)]=\sum_{k=0}^\infty\frac{e^\lambda(-1)^k}{k!}\bar\nu^{\ast k} \text{ with } \lambda:=\nu(\R),\; \bar\nu(A):=\nu(-A).
\]
Therefore, ${\cal F}^{-1}[\phi^{-1}(-\cdot)]\ast {\bf 1}_{(-\infty,t]}$ is
a bounded function, in fact of bounded variation, and the uniform CLT for
the linearized stochastic term follows directly (since $BV$-balls are
universal Donsker classes). The remainder term remains negligible whenever
the inverse bandwidth $h^{-1}$ grows slower than exponentially in $n$.
Choosing for instance $h_n\thicksim \exp(-\sqrt{n})$ yields a pointwise CLT
for $\sqrt{n}(\hat N_n(t)-N(t))$ for all $t\in\R$ if the bias is
negligible, e.g. if $N$ has some positive H\"older regularity at $t$. We do
not pursue a detailed derivation of this specific case here.

\item [Gamma Processes.]
The family of Gamma processes satisfies $X_k\sim \Gamma(\alpha\Delta,\lambda)$, with probability density $\gamma(y; \alpha \Delta, \lambda) = (1/\Gamma(\alpha \Delta))\lambda^{\alpha \Delta} y^{\alpha \Delta -1} e^{-\lambda y}$, L\'evy measure $\nu(dx)=\alpha x^{-1}e^{-\lambda x}{\bf 1}_{\R^+}(x)\,dx$ and characteristic function $\phi(u)=(1-iu/\lambda)^{-\alpha\Delta}$. For simplicity we consider $\lambda=1$ and, in order to satisfy Assumption \ref{gen}(c), we restrict to $\alpha\in(0,1/(2\Delta))$. We denote the density  of $\Gamma(\beta,1)$ by $\gamma_\beta$ and its distribution function by $\Gamma_\beta$. Then
\[{\cal F}^{-1}[\phi^{-1}]={\cal F}^{-1}[(1-iu)^{\alpha\Delta-1}(1-iu)]=\gamma_{1-\alpha\Delta}\ast (\Id+D)
 \]
holds with the differential operator $D$. This is a well known form of the fractional derivative operator of order $\alpha\Delta$. We deduce
\[ {\cal F}^{-1}[\phi^{-1}(-\cdot)]\ast {\bf 1}_{[t,\infty)}=\gamma_{1-\alpha\Delta}(-\cdot)\ast({\bf 1}_{[t,\infty)}-\delta_t).
\]
Hence, for $t>0$ the asymptotic variance of Theorem \ref{Thm1} is given by
\[ \Sigma_{t,t}=\int_0^\infty (1-\Gamma_{1-\alpha\Delta}(t-x)-\gamma_{1-\alpha\Delta}(t-x))^2\gamma_{\alpha\Delta}(x)\,dx.
\]
Note that the integrand has poles of order $(\alpha\Delta)^2$ at $x=t$ and of order $1-\alpha\Delta$ at $x=0$ such that the variance is finite if and only if $\alpha\Delta<1/2$ and $t\not=0$. So, in this case, Assumption \ref{gen}c) prevents $\Sigma_{tt}$ from being infinite.

Moreover, the Gamma process case can serve as a basic example for all the
theory that follows. It reveals the problem that standard $L^p$-theory or
non-local Fourier analysis will not be sufficient in this context as different locations of the singular support (the poles) are required to ensure finiteness of $\Sigma_{t,t}$.

\item[Gamma plus Poisson process.]
Let us briefly give a simple counterexample that pseudo-locality of the
deconvolution operator is important. If the L\'evy process is a
superposition of a Gamma process as above with $\alpha\in(0,1/(2\Delta))$
and of an independent Poisson process with intensity $\lambda>0$, the
density $p$ of the increments is given by the convolution of the
$\gamma_{\alpha\Delta}$-density with a $\Poiss(\lambda)$-law and thus has
poles of order $1-\alpha\Delta$ at $x\in\N_0$. On the other hand, the
deconvolution operator is given by
\begin{align*}
{\cal
F}^{-1}[\phi^{-1}(-\cdot)]&=\sum_{k=0}^\infty\frac{e^\lambda(-1)^k}{k!}\delta_{-k}\ast
\gamma_{1-\alpha\Delta}(-\cdot)\ast(\Id-D)\\
&=\sum_{k=0}^\infty\frac{e^\lambda(-1)^k}{k!}
\gamma_{1-\alpha\Delta}(-\cdot-k)\ast(\Id-D).
\end{align*}
As in the pure Gamma case, this shows that $\Sigma_{t,t}$ is finite if and
only if none of the poles at $x=t-k$, $k\in\N_0$, and at $x=k$, $k\in\N_0$,
of the respective functions coincide, which is the case only for
non-integer $t\notin\N_0$. Consequently, we cannot hope even to prove a
pointwise CLT with rate $1/\sqrt{n}$ at integers $t$. This case that
singularities are just translated by convolution with point measures is
excluded by the regularity requirement for $x\nu$ in Assumption
\ref{gen}(b).

\item[Self-Decomposable Processes.]
We finally consider the class of self-decomposable processes, cf.
\cit{Sa99}, Section 15, which contains all Gamma processes. For any
pure-jump self-decomposable process we have $\nu(dx)=k(x)/\abs{x}\,dx$ with
a unimodal $k$-function increasing on $(-\infty,0)$ and decreasing on
$(0,\infty)$. If the limits $k(0-)$ and $k(0+)$ of $k$ at zero are finite,
then $k$ is a function of bounded variation and so is $\sgn(x)k(x)$, the
density of $x\nu$. The moment condition of Assumption \ref{gen}(a) in
particular implies $\sgn(x)k(x)\in L^1(\R)$ which yields Assumption
\ref{gen}(b). It is quite remarkable that the probabilistic property of
self-decomposability implies the analytic property of pseudo-locality for
the deconvolution operator.

For the characteristic function of self-decomposable processes we have $\abs{\phi(u)}\gtrsim (1+\abs{u})^{-\alpha\Delta}$ with $\alpha=k(0-)+k(0+)$, which follows exactly as the proof of Lemma 2.1 in \cit{T11}. The latter is the counterpart to Lemma 53.9 in \cit{Sa99}, where an upper bound of the same order times a logarithmic factor is shown. We conclude that Assumption \ref{gen}(c) translates to the condition $\alpha<1/(2\Delta)$.
\end{description}

We note that Assumption \ref{gen}(a) and \ref{gen}(b) remain true under superposition of independent L\'evy processes and we collect the findings in an explicit statement.

\begin{proposition} \label{PropEx}
Assumption \ref{gen} is satisfied for
\begin{enumerate}
\item a compound Poisson process whenever the jump law has a density $\nu$ such that $x\nu$ is of bounded variation and
$\nu$ has a finite $(2+\eps)$-moment,
\item a Gamma process with parameters $\alpha\in(0,1/(2\Delta))$ and $\lambda>0$,
\item a pure-jump self-decomposable process whenever its $k$-function satisfies $\int \max(1,\abs{x}^{1+\eps})k(x)\,dx<\infty$ and $k(0-)+k(0+)<1/(2\Delta)$,
\item and for any L\'evy process which is a sum of independent compound Poisson and self-decomposable processes of the preceding types.
\end{enumerate}
\end{proposition}

\subsection{Extensions and perspectives}

There are many directions for further investigation. As from the classical
Donsker result, concrete statistical inference procedures, like L\'{e}vy-analogues of the classical Kolmogorov-Smirnov-tests and
corresponding confidence bands, can be derived from Theorem \ref{Thm1}. Also extensions to uniform CLTs for more
general functionals than just for the distribution function are highly
relevant. A question of particular interest in the area of statistics for
stochastic processes is whether one can allow for high-frequency observation
regimes $\Delta_n \to 0$. As discussed above, decreasing $\Delta \to 0$ renders the inverse
problem more regular, as Assumption \ref{gen}(c) is then easier to satisfy. Since we use the central limit theorem for triangular
arrays in our proofs, allowing $\Delta$ to depend on $n$ should not pose a principal difficulty, but doing so in a sharp way may not only require an estimator based on the second derivative of
$\log(\phi_n)$, but also extra care in controlling all terms uniformly in $n$,
and is beyond the scope of the present paper.

Another issue of statistical relevance is the question of efficiency, which we
briefly address here. Our plug-in estimation method is quite natural and should
have asymptotic optimality properties as the empirical distribution function
has for the classical i.i.d. case. This is also in line with the result by
\cit{KV11} who show that the tangent space of the class of infinitely divisible
distributions with positive Gaussian part is nonparametric to the effect that
the estimation of linear functionals $\int g\,dP$ of $P$ (but not $\nu$ as in
our case) by empirical means is asymptotically efficient. Indeed, a formal
derivation indicates that the pointwise asymptotic variance of our estimator
$\hat N_n(t)$ coincides with the semiparametric Cram\'er-Rao information bound (see \cit{VW96}, Chapter 3.11, for the relevant definitions). Let us
restrict here to the case $t<0$ and assume that the observation law $P_\nu$ has
a Lebesgue density $p_\nu$.

Perturbing the L\'evy measure $\nu$ in direction of an $L^1$-function $h$, we obtain by differentiating in the Fourier domain the score function (the derivative of the log-likelihood)
\[ \dot\ell_\nu(h):=\frac{d}{d\eps} \frac{p_{\nu+\eps h}}{p_\nu}\Big|_{\eps=0}=\frac{{\cal F}^{-1}\Big[\phi_\nu(u)\int (e^{iux}-1)\,h(dx)\Big]}{p_\nu} =\frac{p_\nu\ast(h-\lambda_h\delta_0)}{p_\nu}
\]
with $\lambda_h=\int h$. This yields the Fisher information at measure $\nu$ in direction $h$ as
\[
\scapro{I(\nu)h}{h} :=\E_\nu[\dot\ell_\nu(h)^2]=\int \Big(\frac{p_\nu\ast(h-\lambda_h\delta_0)(x)}{p_\nu(x)}\Big)^2\,P_\nu(dx).
\]
On the other hand, we aim at estimation of the functional $\nu\mapsto N(t)$ whose derivative in direction $h$ by linearity is given by $H(t)=\scapro{{\bf 1}_{(-\infty,t]}}{h}$ (interpreting $\scapro{\cdot}{\cdot}$ as a dual pairing). The semi-parametric Cram\'er-Rao lower bound is then $\sup_{h}\frac{H(t)^2}{\scapro{I(\nu)h}{h}}$, maximising the parametric bound over all sub-models $(\nu+\eps h)_{\eps\in\R}$. The supremum is formally attained at $h^\ast=I(\nu)^{-1}{\bf 1}_{(-\infty,t]}$ with value $\scapro{{\bf 1}_{(-\infty,t]}}{h^\ast}$. The maximiser can be expressed explicitly using the deconvolution operator:
\[ h^\ast={\cal F}^{-1}[\phi^{-1}]\ast \Big\{ p_\nu\times\Big({\cal F}^{-1}[\phi^{-1}(-u)]\ast{\bf
1}_{(-\infty,t]}-{\cal F}^{-1}[\phi^{-1}(-u)]\ast{\bf
1}_{(-\infty,t]}(0)\Big)\Big\}.
\]
Resuming the formal calculus and noting that ${\cal F}^{-1}[\phi^{-1}(-u)]$ is the formal adjoint of  ${\cal F}^{-1}[\phi^{-1}]$, we find the explicit Cram\'er-Rao bound
\[ \int {\bf 1}_{(-\infty,t]}(x)h^\ast(x)\,dx=\int\Big({\cal F}^{-1}[\phi^{-1}(-u)]\ast{\bf 1}_{(-\infty,t]}\Big)(x) p_\nu(x)\Big({\cal F}^{-1}[\phi^{-1}(-u)]\ast{\bf
1}_{(-\infty,t]}\Big)(x)\,dx,
\]
which is exactly equal to the asymptotic variance $\Sigma_{t,t}$ from Theorem
\ref{Thm1}. We have used here that ${\cal F}^{-1}[\phi^{-1}(-u)]\ast{\bf
1}_{(-\infty,t]}(X)$ is centred, cf. \eqref{MainStochTermII} below.

The hardest parametric subproblem of our general semi-parametric estimation
problem is thus given by perturbing $\nu$ in direction of $h^\ast$. The lower
bound for the variance equals exactly the asymptotic variance of our estimator.
Let us nevertheless emphasize that this formal derivation of the Cram\'er-Rao
lower bound does not justify asymptotic efficiency in a completely rigorous
manner: for this one would have to establish the regularity of the statistical
model and $h^\ast\in L^1(\R)$, which appears to require an even finer analysis
of the main terms than our Donsker-type result. The complete proof remains a
challenging open problem.

\section{Proof of Theorem \ref{Thm1}}

The remainder of this article is devoted to the proof of Theorem \ref{Thm1}, which is split into the separate proofs of convergence of the finite-dimensional distributions and of tightness. We shall repeatedly use the following auxiliary lemma.

\begin{lemma}\label{LemmaLevyProp}
Suppose $\gamma=0$. Then Assumption \ref{gen} implies:
\begin{enumerate}
\item The measure $xP=xP(dx)$ has a bounded Lebesgue density on $\R$.
\item $(\phi^{-1})'\in L^2(\R) \cap L^\infty(\R)$ as well as $\abs{\phi^{-1}(u)}\lesssim (1+\abs{u})^{(1-\eps)/2}$ for all $u \in \mathbb R$;
\item $m(u):=\phi^{-1}(-u)(1+iu)^{(-1+\eps)/2}$ is a Fourier multiplier on every Besov space $B^s_{p,q}(\R)$ with $s\in\R$, $p,q\in[1,\infty]$; that is convolution with ${\cal F}^{-1}m$ is continuous from $B^s_{p,q}(\R)$ to $B^s_{p,q}(\R)$.
\end{enumerate}
\end{lemma}

\begin{proof}
\mbox{}
\begin{enumerate}
\item From (\ref{EqIdent}) with $\gamma=0$ we see
\[ {\cal F}[ixP](u)=\phi'(u)=i\Delta{\cal F}[x\nu](u){\cal F}P(u) \Rightarrow xP=\Delta (x\nu)\ast P
\]
and thus with $x\nu$ (Assumption \ref{gen}(b)) also $xP$ has a Lebesgue density $xp(x)$ with $\norm{xp}_\infty\le \Delta\norm{x\nu}_\infty$.
\item From Assumption \ref{gen}(b) and $\gamma=0$ we deduce $\abs{\psi'(u)}\lesssim (1+\abs{u})^{-1}$ and thus $\norm{(1+\abs{u})^\eps(\phi^{-1})'}_{L^2}\lesssim \norm{\phi^{-1}(1+\abs{u})^{-1+\eps}}_{L^2}<\infty$ by Assumption \ref{gen}(c). This implies
  \begin{eqnarray*} 
  \abs{\phi^{-1}(u)} &\le& 1+\int_0^u \abs{(\phi^{-1})'(v)}\,dv\lesssim 1+\norm{(1+\abs{v})^\eps(\phi^{-1})'}_{L^2}\norm{(1+\abs{v})^{-\eps}{\bf 1}_{[0,u]}}_{L^2} \\
&\lesssim& (1+|u|)^{(1/2)-\varepsilon} \lesssim (1+\abs{u})^{(1-\eps)/2},
\end{eqnarray*}
 and then also $|(\phi^{-1})'|(u) \lesssim |\phi^{-1}(u)| |\psi'(u)| \lesssim 1$, so $(\phi^{-1})' \in L^\infty(\R)$.
\item The Fourier multiplier property of $m$ follows from  the Mihlin multiplier theorem
for Besov spaces (see e.g. \cit{Tr10} and particularly the scalar version of Cor. 4.11(b) in \cit{GW03}): because of (b) the function $m$ is bounded and satisfies
\[ \abs{u m'(u)}\lesssim \abs{u m(u)}(1+\abs{u})^{-1}\lesssim 1.\]
Consequently, the conditions of Mihlin's multiplier theorem are fulfilled and $m$ is a Fourier multiplier on all Besov spaces $B^s_{p,q}(\mathbb R)$.
\end{enumerate}
\end{proof}

\subsection{Convergence of the Finite-Dimensional Distributions}
Denote by $H^s(\mathbb R), s \in \mathbb R,$ the standard $L^2$-Sobolev spaces
with norm $\|h\|_{H^s}:=\|\mathcal Fh(u) (1+|u|)^s\|_{L^2}$.
\begin{definition} \label{admd}
We say that a function $g\in L^\infty(\R) \cap L^2(\R)$ is {\em admissible} if
\begin{enumerate}
\item $g$ is Lipschitz continuous in a neighbourhood of zero,
\item we can split $g=g^c+g^s$ into functions $g^c \in H^1(\mathbb R),\, g^s \in L^1(\mathbb
R)$, satisfying $\max(\abs{{\cal F}[g^s](u)},\abs{{\cal
F}[xg^s](u)})\lesssim (1+\abs{u})^{-1}$ for all $u \in \mathbb R$.
\end{enumerate}
\end{definition}

\begin{lemma} \label{adm}
The functions $g_t$ from \eqref{EqNhat} as well as all finite linear
combinations $\sum_i\alpha_ig_{t_i}$ with $\alpha_i\in\R$, $t_i\not=0,$ are
admissible. Moreover, we can choose $g_t^c, g_t^s$ in such a way that
$$\|g^c_t\|_{H^1}\lesssim (1+|t|)^{-1/2},\quad |{\cal F}g^s_t(u)|
\lesssim (1+|u|)^{-1}(1+|t|)^{-1} ~\text{ and }~ |{\cal F}[xg^s_t](u)|
\lesssim (1+|u|)^{-1},
$$
the inequalities holding with constants independent of $u\in\R$,
$t\in\R\setminus(-\zeta,\zeta)$ for $\zeta>0$ fixed.
\end{lemma}
\begin{proof}
First note that all properties of admissible functions remain invariant under
finite linear combinations and reflection $g\mapsto g(-\cdot)$. It thus
suffices to check that $g_t$, $t<0$, is admissible. Let $\chi\in
C^\infty((-\infty,0])$ be a smooth function with $\chi(0)=1$ and $\chi,
\chi'$ both bounded and integrable on $(-\infty, 0]$, for instance $\chi(x)=e^x1_{(-\infty,
0]}$. Decompose $g_t=g_t^c+g_t^s$ with $$g_t^c(x)=g_t(x)(1-\chi(x-t)), ~~
g_t^s(x)=g_t(x)\chi(x-t); ~~~\text{for }x\le t,$$ and both equal to zero for
$x>t$. Then $g_t^c \in L^2(\mathbb R)$ and its (weak) derivative is
$$(g_t^c)'(x) = -x^{-2}(1-\chi(x-t))1_{(-\infty, t]}(x) +
x^{-1}(1-\chi(x-t))'1_{(-\infty, t]}(x) \in L^2(\mathbb R),$$ so $g_t^c\in
H^1(\mathbb R)$. The functions $g_t^s, xg_t^s$ are both integrable since $\chi$
is. The (weak) derivatives of $xg_t^s$ and $g_t^s$ are $\chi'(x-t)1_{(-\infty,
t)} -\delta_t$ and $-x^{-2}\chi(x-t)1_{(-\infty, t]} + x^{-1}
\chi'(x-t)1_{(-\infty, t)} - t^{-1}\delta_t$, respectively, with point measures
$\delta_t$. So, both functions are of bounded variation and their Fourier
transforms are bounded by $(1+\abs{u})^{-1}$ up to multiplicative constants.
Finally, observe that $g_t$ is constant and thus Lipschitz near zero, so that
$g_t$ is admissible.

For the second claim we again only consider $t<0$ and first observe, $\chi$
being bounded, that  $$\|g^c_t\|^2_{L^2} \lesssim \int_{-\infty}^t |x|^{-2} \sim
|t|^{-1}$$ as $t \to -\infty$. Likewise, using the explicit form of $(g_t^c)'$,
we see $$\|g^c_t\|_{H^1} \lesssim \|g^c_t\|_{L^2}+\|(g_t^c)'\|_{L^2} \lesssim
(1+|t|)^{-1/2}.$$ For $g_t^s=x^{-1}1_{(-\infty, t]} \chi(x-t)$ we see
$\|g^s_t\|_{L^1} \le t^{-1} \|\chi\|_{L^1}$, and the total variation of the
derivative of $g_t^s$ is bounded by $t^{-2}\|\chi\|_{L^1}+t^{-1}\|\chi'\|_{L^1}
+ t^{-1}$.  We conclude that $\abs{{\cal F}g_t^s(u)}\lesssim
(1+\abs{u})^{-1}(1+\abs{t})^{-1}$ holds. The same argument gives a bound
independent of $t$ for $|{\cal F}[xg^s_t](u)|$, thus completing the proof.
\end{proof}

\begin{theorem}\label{ThmCLT}
Suppose Assumption \ref{gen} is satisfied, $g$ is admissible and $h_n\thicksim n^{-1/2}(\log n)^{-\rho}$ for some $\rho>1$. Then setting
\[ \hat N_n(g):=\frac{1}{i\Delta}\int_{\R}g(x){\cal F}^{-1}[(\phi_n'/\phi_n) {\cal F}K_{h_n}](x)\,dx,\quad N(g):=\int g(x)x\nu(dx)
\]
(with some abuse of notation $N(t)=N(g_t)$ etc.), we have asymptotic normality,
\[\sqrt n \Big(\hat N_n(g)-N(g)\Big) \to^{\mathcal L}N(0,\sigma_g^2) \]
as $n \to \infty$ with finite variance
\[
\sigma_g^2 = (i\Delta)^{-2}\int_{\R}\Big({\cal F}^{-1}
[{\cal F}g(u)\phi^{-1}(-u)](x)ix+{\cal F}^{-1}[{\cal F}g(u)(\phi^{-1})'(-u)](x)\Big)^2\,P(dx).
\]
\end{theorem}

\begin{corollary}\label{CorFidi}
Under the assumptions of the preceding theorem the finite-dimensional distributions of the processes $(\sqrt{n}(\hat N_n(t)-N(t)),t\in\R\setminus\{0\})$ converge to $\mathbb G^\varphi$ as $n \to \infty$, where $\mathbb G^\varphi$ is a centered Gaussian process, indexed by $\R\setminus\{0\}$, with covariance structure given by \eqref{gausscov} for $t,s\in\R\setminus\{0\}$.
\end{corollary}

\begin{proof}
This follows directly by the Cram\'er-Wold device applied to any finite subfamily of $(g_t,\,t\in\R\setminus\{0\})$, using the preceding lemma and theorem.
\end{proof}

The remaining part of this subsection is devoted to the proof of Theorem \ref{ThmCLT}.

\subsubsection{Discarding the drift $\gamma$}\label{SecDrift}

We shall show that we may assume $\gamma=0$ in the sequel. To see this, observe that shifting $X_k\mapsto \tilde X_k=X_k+\gamma$ leads to the shift in the empirical quotient
\[\phi_n'(u)/\phi_n(u)\mapsto \tilde\phi_n'(u)/\tilde\phi_n(u)=(e^{iu\gamma}\phi_n)'(u)/(e^{iu\gamma}\phi_n(u))=i\gamma+\phi_n'(u)/\phi_n(u) \]
and the true quotient also satisfies $\tilde\phi'(u)/\tilde\phi(u)=i\gamma+\phi'(u)/\phi(u)$. In $\hat N_n(g)-N(g)$ this shift thus induces the error
\begin{align*}
\babs{\frac{1}{i\Delta}\int_{\R}g(x){\cal F}^{-1}[i\gamma({\cal F}K_h-1)](x)\,dx}
&=\frac{|\gamma|}{\Delta}\babs{\int_{\R}(g(x)-g(0))K_h(x)\,dx}\\
&\lesssim \int_{\R} \norm{g}_{\text{Lip}(0)}\abs{x} \abs{K_h(x)}\,  dx +  \int_{[-\delta, \delta]^c} \|g\|_\infty|K_h(x)|dx\\
&\lesssim \int_{\R} \abs{x} h^{-1}(1\wedge\abs{x/h}^{-\beta})\,dx + \int_{[-\delta/h, \delta/h]}(1+|u|)^{-\beta}du \lesssim h,
\end{align*}
where we have used the Lipschitz constant of $g$ in a $\delta$-neighbourhood of
zero and \eqref{EqKProp} with $\beta>2$. By the choice of $h=h_n$ this error is
of order $O(h_n)=o(n^{-1/2})$ and thus negligible in the asymptotic distribution of $\sqrt n (\hat N(g) -N(g))$, and we note that this bound is uniform in
all $g$ satisfying the admissibility conditions with uniform constants.
Henceforth, without loss of generality, we shall only consider the case
$\gamma=0$.

\subsubsection{Approximation error}


By approximation error we understand here the deterministic 'bias' term
$$\frac{1}{2\pi i\Delta}\int_{\R}{\cal F}g(-u)\frac{\phi'(u)}{\phi(u)}
{\cal F}K_h(u)du - \frac{1}{2\pi i\Delta}\int_{\R}{\cal
F}g(-u)\frac{\phi'(u)}{\phi(u)}du$$ induced by the spectral cutoff with
${\cal F}K_h$. We use Assumption \ref{gen}(b), i.e. that
$\abs{\psi'(u)}=\abs{{\cal F}[x\nu](u)} \lesssim (1+\abs{u})^{-1}$. Moreover,
we split $g=g^c+g^s$ and treat the bias of each term separately.

For the term involving $g^s$, using the Lipschitz continuity and boundedness of
${\cal F}K$ (due to \eqref{EqKProp} with $\beta>2$),
\begin{align*}
\frac{1}{2\pi \Delta}\babs{\int_{\R}{\cal F}g^s(-u)\frac{\phi'(u)}{\phi(u)} (1-{\cal F}K_h)(u)\,du}
&\lesssim \int_{\R}(1+\abs{u})^{-1}\abs{\psi'(u)}\abs{1-{\cal F}K(hu)}\,du\\
&\lesssim \int_{\R}(1+\abs{u})^{-2}\min(h\abs{u},1)\,du\\
&\lesssim h\log(h^{-1}).
\end{align*}
For $g^c$ we have by the Cauchy-Schwarz inequality
\begin{align*}
\frac{1}{2\pi \Delta}\babs{\int_{\R}{\cal F}g^c(-u)\frac{\phi'(u)}{\phi(u)} (1-{\cal F}K_h)(u)\,du}
&\lesssim \int_{\R}(1+|u|)|{\cal F}g^c(-u)|(1+|u|)^{-2}h|u|\,du \\
&\lesssim h\|g^c\|_{H^1} \Big(\int_\mathbb R (1+|u|)^{-2}du\Big)^{1/2} \thicksim h
\end{align*}
Combining these two estimates, and since $h=h_n=o(n^{-1/2}\log(n)^{-1})$, we
conclude that the bias term is of negligible order $o(n^{-1/2})$ in the asymptotic distribution of $\sqrt n (\hat N(g)- N(g))$.

\subsubsection{Main stochastic term}\label{SecStochTerm}

Linearising the error in the quotient $\phi_n'/\phi_n$ we identify two major stochastic terms:
\[ \frac{\phi_n'(u)}{\phi_n(u)}-\frac{\phi'(u)}{\phi(u)} =\phi^{-1}(u)(\phi_n'-\phi')(u)+(\phi^{-1})'(u)(\phi_n-\phi)(u) + R_n(u)
\]
with remainder
\begin{equation}\label{EqRemainder}
 R_n(u):=\Big(1-\frac{\phi_n(u)}{\phi(u)}\Big) \Big(\frac{\phi_n'(u)}{\phi_n(u)}-\frac{\phi'(u)}{\phi(u)}\Big)
\end{equation}
where we used the identity $\phi^{-1}\phi'+(\phi^{-1})'\phi=(\phi^{-1}\phi)'=0$. Discarding the remainder term for the time being, we study the linear centered term
\begin{align}
&\frac{1}{2\pi i\Delta}\int_{\R}{\cal F}g(-u){\cal F}K_h(u)\Big(\phi^{-1}(u)(\phi_n'-\phi')(u)+(\phi^{-1})'(u)(\phi_n-\phi)(u)\Big) \,du\nonumber \\
&=\frac{1}{2\pi i\Delta}\int_{\R}{\cal F}g(-u){\cal F}K_h(u)\Big(\phi^{-1}(u)\phi_n'(u)+(\phi^{-1})'(u)\phi_n(u)\Big) \,du \nonumber\\
&=\frac{1}{2\pi i\Delta}\int_{\R}{\cal F}g(-u){\cal F}K_h(u)\Big(\phi^{-1}(u){\cal F}[ixP_n](u)+(\phi^{-1})'(u){\cal F}[P_n](u)\Big) \,du \nonumber\\
&=\frac{1}{i\Delta}\int_{\R}\Big({\cal F}^{-1}
\Big[\phi^{-1}(-u){\cal F}g(u){\cal F}K_h(-u)\Big](x)ix+{\cal F}^{-1}\Big[(\phi^{-1})'(-u){\cal F}g(u){\cal F}K_h(-u)\Big](x)\Big) \,P_n(dx). \label{MainStochTermII}
\end{align}
These manipulations are justified by standard Fourier analysis of finite
measures, using the compact support of $\mathcal F K_h$ and of $P_n$ as well as
that $(1+\abs{u})^{-1}\phi^{-1}(u), {\cal F}g, (\phi^{-1})'$ are all in
$L^2(\R)$ (by virtue of Assumption \ref{gen}(c), admissibility of $g$, Lemma
\ref{LemmaLevyProp}(b)).

Thus, the central limit theorem for triangular arrays under Lyapounov's
condition (e.g. Theorem 28.3 combined with (28.8) in \cit{Ba96}) applies to the
standardised sums if
 \begin{equation}\label{MainCond}
 \sup_{h\in(0,1)} \int_{\R}\babs{{\cal F}^{-1}
\Big[\phi^{-1}(-u){\cal F}g(u){\cal F}K_h(-u)\Big](x)ix+{\cal F}^{-1}\Big[(\phi^{-1})'(-u){\cal F}g(u){\cal F}K_h(-u)\Big](x)}^{2+\eps} P(dx)
\end{equation}
is finite.

We use the decomposition $g=g^c+g^s$ and deal with $g^c$ first. We have from the Cauchy-Schwarz inequality, Assumption \ref{gen}(c) and admissibility of $g$
\begin{equation} \label{cterm}
\int_\mathbb R |\mathcal F[g^c](u)| |\phi^{-1}(-u)|du \le \|g^c\|_{H^1} \|\mathcal F^{-1}[\phi^{-1}]\|_{H^{-1}} < \infty.
\end{equation}
Since also $\sup_{h>0,u}\abs{{\cal F}K_h(u)}\le \|K\|_{L^1}<\infty$ we have $\mathcal F[g^c] \phi^{-1}(-\cdot){\cal F}K_h \in L^1(\mathbb R)$  and thus
\[ \sup_{h\in(0,1)}{\cal F}^{-1}
[\phi^{-1}(-u){\cal F}g^c(u){\cal F}K_h(-u)]\in L^\infty(\R).
\]
The integral over the first term in \eqref{MainCond} with $g^c$ replacing $g$ is thus finite in view of $\int \abs{x}^{2+\eps}P(dx)<\infty$ by Assumption \ref{gen}(a).

For the singular part we remark $\abs{({\cal F}K_h)'(u)}\le\norm{xK_h}_{L^1}\lesssim h$ as well as (by Assumption \ref{gen}(b)) $\abs{(\phi^{-1})'(u)}=\Delta\abs{\psi'(u)\phi^{-1}(u)}\lesssim (1+\abs{u})^{-1}\abs{\phi^{-1}(u)}$.
We conclude uniformly in $h$, using admissibility of $g$,
\[ \abs{(\phi^{-1}(-\cdot){\cal F}g^s{\cal F}K_h(-\cdot))'(u)}\lesssim \abs{\phi^{-1}(u)}(1+\abs{u})^{-1}.
\]
By Assumption \ref{gen}(c) and the Sobolev embedding this implies
\begin{equation}\label{EqIneqgs}
\sup_h {\cal F}^{-1}
[\phi^{-1}(-u) {\cal F}g^s(u){\cal F}K_h(-u)](x)(1+ix)\in H^\eps(\R)\subset L^{2+\eps}(\R).
\end{equation}
Using Lemma \ref{LemmaLevyProp}(a) and $\abs{x}^{2+\eps}\le \abs{x}\abs{1+ix}^{2+\eps}$, also the integral over the first term in \eqref{MainCond} with $g^s$ replacing $g$ is finite.

For the integral over the second term in \eqref{MainCond} we recall
$\sup_{h>0,u}\abs{{\cal F}K_h(u)}\le \|K\|_{L^1}<\infty$ and that $\mathcal Fg,
(\phi^{-1})'$ are both in $L^2(\mathbb R)$ to deduce $\abs{{\cal F}g(u){\cal
F}K_h(-u)(\phi^{-1})'(-u)} \in L^1(\mathbb R)$ by the Cauchy-Schwarz
inequality. By Fourier inversion ${\cal F}^{-1}[{\cal F}g(u){\cal
F}K_h(-u)(\phi^{-1})'(-u)] \in L^\infty$ holds, and since $P$ is a probability
measure, also the integral over the second term is finite.

Altogether we have shown that under our conditions the main stochastic error
term is asymptotically normal with rate $1/\sqrt{n}$ and mean zero. For
$n\to\infty$ the variances converge to $\sigma_g^2$, which follows from ${\cal
F}K_{h_n}\to 1$ pointwise  and uniform integrability by bounded
($2+\eps$)-moments.

\subsubsection{Remainder term}\label{SecRemainder}

In what follows $\Pr$ stands for the usual product probability measure $P^\mathbb N$ describing the joint law of $X_1, X_2, \dots$, and $Z_n=O_P(r_n)$ means that $r_n^{-1}Z_n$ is bounded in $\Pr$-probability. We show that the remainder term is $O_P(r_n)$ for some $r_n=o(n^{-1/2})$, and therefore negligible in the asymptotic distribution of $\sqrt n (\hat N(g) -N(g))$. 

From Theorem 4.1 of \cit{NR09} we have for any $\delta>0$, using the finite
$(2+\eps)$-moment property of $P$ from \eqref{EqMoments},
\[\sup_{\abs{u}\le U}\Big(\abs{\phi_n(u)-\phi(u)}+\abs{\phi_n'(u)-\phi'(u)}\Big)=O_P(n^{-1/2}(\log U)^{1/2+\delta}).
\]
This implies in particular, using
\begin{equation} \label{lbd}
\inf_{|u| \le h_n^{-1}}|\phi(u)| \gtrsim \inf_{|u| \le h_n^{-1}}(1+|u|)^{-1/2} \gtrsim  \sqrt h_n \gtrsim n^{-1/4} (\log n)^{-\rho/2}
\end{equation}
from Lemma \ref{LemmaLevyProp}(b), that for any constant $0<\kappa<1$,
\begin{eqnarray*}
&&\Pr \left(\left|\frac{1}{\phi_n(u)}\right| < \left|\frac{\kappa}{\phi(u)}\right|~\text{for some } u \in [-h_n^{-1}, h_n^{-1}]
\right) \\
&& = \Pr \left(\left|\frac{\phi_n(u)}{\phi(u)} \right| > \kappa^{-1}~\text{for some } u \in [-h_n^{-1}, h_n^{-1}] \right) \\
&&\le \Pr \left(\left|\frac{\phi_n(u)-\phi(u)}{\phi(u)} \right| > (\kappa^{-1}-1)~\text{for some }  u \in [-h_n^{-1}, h_n^{-1}] \right) \\
&& \le \Pr \left(\sup_{|u| \le h_n^{-1}}\left|\phi_n(u)-\phi(u) \right| \gtrsim n^{-1/4} (\log n)^{-\rho/2} \right) \to 0
\end{eqnarray*}
as $n \to \infty$, in other words, on events of probability approaching one,
$\phi_n^{-1}$ decays no faster than $\phi^{-1}$ uniformly on increasing sets
$[-h_n^{-1}, h_n^{-1}]$.

Now to control the remainder term \eqref{EqRemainder} we use $\supp({\cal F}K_h)\subset[-h^{-1},h^{-1}]$ and distinguish each term of the decomposition $g=g^s+g^c$. First, using $\abs{{\cal F}g^s(u)}\lesssim (1+\abs{u})^{-1}$, Lemma \ref{LemmaLevyProp}(b) and Assumption \ref{gen}(c) we see
\begin{align*}
&\babs{\int_{-h^{-1}}^{h^{-1}} {\cal F}g^s(-u){\cal F}K_h(u)R_n(u)\,du}\\
&= O_P\Big( \int_{-h^{-1}}^{h^{-1}} (1+\abs{u})^{-1}n^{-1}(\log h^{-1})^{1+2\delta} \abs{\phi^{-1}(u)} \big(\abs{\phi(u)^{-1}}+\abs{(\phi^{-1})'(u)}\big)\,du\Big)\\
&=O_P\Big(n^{-1}(\log h^{-1})^{1+2\delta}h^{2\eps-1}\int (1+\abs{u})^{-2+2\eps}\abs{\phi(u)}^{-2}du\Big)\\
&=O_P\big(n^{-1}(\log h^{-1})^{1+2\delta}h^{2\eps-1}\big).
\end{align*}
For the nonsingular part we have likewise, using the Cauchy-Schwarz inequality, $g^c \in {H^1}(\mathbb R)$, (\ref{lbd}) and Assumption \ref{gen}(c),
\begin{align*}
\babs{\int_{-h^{-1}}^{h^{-1}} {\cal F}g^c(-u){\cal F}K_h(u)R_n(u)\,du}
&= O_P\Big( n^{-1}(\log h^{-1})^{1+2\delta} \Big(\int_{-h^{-1}}^{h^{-1}} (1+\abs{u})^{-2} \abs{\phi(u)}^{-4}\,du\Big)^{1/2}
\Big) \\
&= O_P\Big( n^{-1}(\log h^{-1})^{1+2\delta} h^{-1/2} \|\phi^{-1} (1+|u|)^{-1}\|_{L^2}\Big).
\end{align*}
Consequently, the remainder term is negligible because $h_n^{-1+2\eps}(\log h_n^{-1})^{1+2\delta}=o(n^{1/2})$. Note that this gives in fact uniform $o_P(n^{-1/2})$-control of the remainder term for all $g$ that satisfy the admissibility bounds uniformly.

\subsection{Tightness of the Linear Term}

We study the linear part (\ref{MainStochTermII}) and introduce the empirical process
\begin{eqnarray}\label{EqNun}
\nu_n^\phi(t)&:=& \sqrt n\frac{1}{i\Delta}\int_{\R}\Big({\cal F}^{-1}
\Big[\phi^{-1}(-u){\cal F}g_t(u){\cal F}K_{h_n}(-u)\Big](x)ix +  \\
&& ~~~~~~~~~~~~~~~ {\cal F}^{-1}\Big[(\phi^{-1})'(-u){\cal F}g_t(u){\cal F}K_{h_n}(-u)\Big](x)\Big) (P_n-P)(dx),\quad \abs{t}\ge\zeta>0.
\nonumber \end{eqnarray}
Recall that this process is centered even without subtracting $P$. Moreover, since $\sup_{|t|\ge \zeta} \|g_t\|_{L^2}<\infty$, the arguments after (\ref{MainStochTermII}) imply that $\nu_n^\phi$ is a (possibly non-measurable) random element of the space $\ell^\infty((-\zeta, \zeta)^c)$ of bounded functions on $(-\infty, -\zeta] \cup [\zeta, \infty)$ (the complement of $(-\zeta, \zeta)$ in $\mathbb R$) equipped with the uniform norm $\|\cdot\|_{(-\zeta, \zeta)^c}$.

\subsubsection{Pregaussian limit process}

Theorem \ref{Thm1} will follow if we show that $\nu_n^\varphi$ converges to $\mathbb G^\varphi$ in law in $\ell^\infty((-\zeta, \zeta)^c)$. For this statement to make sense we have to show first that $\mathbb G^\varphi$ defines a proper Borel random variable in $\ell^\infty((-\zeta, \zeta)^c)$, which is implied by the following more general result. Recall that any Gaussian process $\{\mathbb G(t)\}_{t \in T}$ induces its intrinsic covariance metric $d^2(s,t)=E(\mathbb G(s)- \mathbb G(t))^2$ on the index set $T$.

\begin{theorem} \label{scgauss}
Grant Assumption $\ref{gen}$. The Gaussian process $\{\mathbb G^\phi (t)\}_{t: |t| \ge \zeta}$ with covariance given by \eqref{gausscov} admits a version, still denoted by $\mathbb G^\phi$, which has uniformly continuous sample paths almost surely for the intrinsic covariance metric of $\mathbb G^\phi$, and which satisfies $\sup_{t:|t|\ge \zeta}|\mathbb G^\phi(t)|<\infty$ almost surely.
\end{theorem}

The proof moreover implies that $(-\zeta, \zeta)^c$ is totally bounded in the metric $d$. Therefore (a version of) $\mathbb G^\varphi$ concentrates on the separable subspace of $\ell^\infty((-\zeta, \zeta)^c)$ consisting of bounded $d$-uniformly continuous functions on $(-\zeta, \zeta)^c$, from which we may in particular conclude that $\mathbb G^\varphi$ defines a Borel-random variable in that space, and hence is also a Borel random variable in the ambient space $\ell^\infty((-\zeta, \zeta)^c)$.

Next to Dudley's entropy integral, the main tool in the proof of Theorem \ref{scgauss} is the following bound for the pseudo-differential operator ${\cal F}^{-1}[\phi^{-1}(-u)]$. For $f \in L^2(\mathbb R)$ we set ${\cal F}^{-1}[\phi^{-1}(-\cdot)] \ast f := {\cal F}^{-1}[\phi^{-1}(-u) {\cal F}f(u)]$ which is well defined at least in $H^{(1-\eps)/2}(\mathbb R)$ in view of Lemma \ref{LemmaLevyProp}. Alternatively, $\|{\cal F}^{-1}[\phi^{-1}(-\cdot)] \ast f \|_{L^2} \lesssim \norm{(1+\abs{u})^{(1-\eps)/2}{\cal F}f(u)}_{L^{2}}$ whenever $f \in H^{(1-\eps)/2}(\R)$, but such an inequality is not sufficient for our purposes. We need a stronger estimate for functions $f$ supported away from the origin, and with the $\|\cdot\|_{L^2}$-norm replaced by the $\|\cdot\|_{2,P}$-norm. Intuitively speaking, and considering the example $f=1_{(s,t]}, s<t<0,$ relevant below, this strengthening is possible since the locations of singularities of $1_{(s, t]}$ and of $P$ (at the origin) are separated away from each other, and since this remains so after application of the pseudo-local operator ${\cal F}^{-1}[\phi^{-1}(-\cdot)]\ast (\cdot)$ to $f$.

\begin{proposition}\label{PropL2PEst}
Grant Assumption \ref{gen} and define $\norm{h}_{2,P}:= (\int h^2dP)^{1/2}$. For $f \in L^2(\R)$ with $\supp(f)\cap(-\delta,\delta)=\varnothing$ for some $\delta>0$ we have
\begin{equation}
\norm{{\cal F}^{-1}[\phi^{-1}(-u)]\ast f}_{2,P} \lesssim \norm{(1+\abs{u})^{1-\eps}{\cal F}f(u)}_{L^{2+4/\eps}(\R)}+\Big(\int \frac{f(y)^2}{1+y^2}\,dy\Big)^{1/2}\label{EqPhiL2}
\end{equation}
provided the right-hand side is finite. The constant in this bound depends only on $\delta$.
\end{proposition}

\begin{proof}
We  shall need the pseudo-differential operator identity
\begin{equation}\label{EqPsiId}
({\cal F}^{-1}[\phi^{-1}(-u)]\ast f)(x)=\Big(\Big(\frac{1}{i\cdot}{\cal F}^{-1}[(\phi^{-1}(-u))']\Big)\ast f\Big)(x),\quad f\in L^2(\R),\;x\notin\supp(f),
\end{equation}
where the right hand side is defined classically. This identity is fundamental for establishing the property of pseudo-locality in a $C^\infty$-framework, see e.g. Theorems 8.8 and 8.9 in \cit{Fo95}. Let us verify this identity here, where $\phi^{-1}\notin C^\infty$. Consider $f \in L^2(\R)$ and $g$ any smooth compactly supported test function such that $\supp(f)\cap\supp(g)=\varnothing$. Then $(f \ast g(-\cdot))(0)=0$ and $f \ast g$ is smooth from which we may conclude that also $x^{-1}(f\ast g(-\cdot))(x)$ (equal to $(f \ast g)'(0)$ at zero) is in $L^2(\R)$ and smooth, and that
\[ {\cal F}\Big[\frac{(f\ast g(-\cdot))(x)}{ix}\Big]'(u)={\cal F}[f\ast g(-\cdot)](u)={\cal F}f(u)\overline{{\cal F}g(u)}.
\]
Plancherel's formula, integration by parts and Fubini's theorem (using $(\phi^{-1})'\in L^2(\R)$ from Lemma \ref{LemmaLevyProp} and the support properties) yield
\begin{align*}
\int ({\cal F}^{-1}[\phi^{-1}(-u)]\ast f)(x)g(x)\,dx &=\frac{1}{2\pi} \int \phi^{-1}(-u){\cal F} f(u)\overline{{\cal F}g(u)}\,du\\
&=\frac{-1}{2\pi} \int (\phi^{-1}(-u))' {\cal F}\Big[\frac{(f\ast g(-\cdot))(x)}{ix}\Big](u)\,du\\
&=\int \frac{{\cal F}^{-1}[(\phi^{-1}(-u))'](x)}{ix} (f\ast g(-\cdot))(-x)\,dx\\
&=\int \Big(\frac{{\cal F}^{-1}[(\phi^{-1}(-u))']}{i\cdot}\ast f\Big)(x)g(x)\,dx.
\end{align*}
In this calculation the boundary terms vanish due to the fast decay of ${\cal F}[(f\ast g(-\cdot))(x)/x]$ ($g$ smooth). Consequently, \eqref{EqPsiId} follows by testing with all $g$ supported near $x$.

We use H\"{o}lder's inequality, the Hausdorff-Young inequality from Fourier analysis, the bound $p(x)\lesssim \abs{x}^{-1}$ from Lemma \ref{LemmaLevyProp}, the pseudo-differential operator identity, again H\"older's inequality, Assumption \ref{gen}(c) and $(\phi^{-1})'\in L^2$ in view of Lemma \ref{LemmaLevyProp} in this order to obtain for $\delta'=\delta/2$:
\begin{align*}
& \babs{\int ({\cal F}^{-1}[\phi^{-1}(-u)]\ast f)^2(x)P(dx)\,dx}
\le \norm{{\cal F}^{-1}[\phi^{-1}(-u)]\ast f}_{L^{2+\eps}(\R)}^2 \norm{p}_{L^{(2+\eps)/\eps}([-\delta',\delta']^c)}\nonumber\\
 &\quad + \norm{{\cal F}^{-1}[\phi^{-1}(-u)]\ast f}_{L^\infty([-\delta',\delta'])}^2 P([-\delta', \delta])\\
&\lesssim \norm{\phi^{-1}(-u){\cal F}f}_{L^{(2+\eps)/(1+\eps)}}^2\norm{xp}_\infty
(\delta')^{-2/(2+\eps)}\\
&\quad +\norm{({\cal F}^{-1}[(\phi^{-1})'(-u)](x)/x)\ast f}_{L^\infty([-\delta',\delta'])}^2\\
&\lesssim \norm{\phi^{-1}(-u)(1+\abs{u})^{-1+\eps}}_{L^2}^2 \norm{(1+\abs{u})^{1-\eps}{\cal F}f(u)}_{L^{2+4/\eps}}^2(\delta')^{-2/(2+\eps)}\\
&\quad +\norm{{\cal F}^{-1}[(\phi^{-1})']}_{L^2}^2\sup_{x\in[-\delta',\delta']} \int \frac{f(y)^2}{(x-y)^2}\,dy\\
&\lesssim \norm{(1+\abs{u})^{1-\eps}{\cal F}f(u)}_{L^{2+4/\eps}}^2+\int \frac{f(y)^2}{1+y^2}\,dy,
\end{align*}
provided $f$ is such that the last line is finite. Take square roots to deduce the asserted inequality with a constant independent of $f$.
\end{proof}

\begin{proof}[Proof of Theorem \ref{scgauss}]
We consider the generalised Brownian bridge process arising as the pointwise weak limit of \eqref{EqNun}, so with ${\cal F}K_h\equiv1$, and further split $g_t=g_t^c+g_t^s$ as in the proof of Lemma \ref{adm}. More precisely, we study the Gaussian process indexed by $(i \Delta)^{-1}$ times
\begin{align}\label{htrig}
h_t(x)&= {\cal F}^{-1}[(\phi^{-1})'(-u){\cal F}g_t(u)](x)+\left({\cal F}^{-1}[\phi^{-1}(-u){\cal F}g_t(u)](x) \right)ix  \\
&={\cal F}^{-1}[(\phi^{-1})'(-u){\cal F}g_t(u)](x)+\left({\cal F}^{-1}[\phi^{-1}(-u){\cal F}g^c_t(u)](x) + {\cal F}^{-1}[\phi^{-1}(-u){\cal F}g^s_t(u)](x)\right)ix, \notag
\end{align}
where $|t|\ge \zeta$. The theorem is thus proved if we show that the class of
functions $\mathcal G = \{(i \Delta)^{-1}h_t: t \in\R\setminus(-\zeta,\zeta)\}$
is bounded in $L^2(P)$ and $P$-pregaussian (cf.~\cit{D99}, Chapter 2, p.92-93).
In Section \ref{SecStochTerm} above we have shown the
$L^{2+\eps}(P)$-boundedness of the same function class, but also involving the
kernel $K_h$. The same proof, replacing ${\cal F}K_h$ just by one, shows that
$\mathcal G$ is even $L^{2+\eps}(P)$-bounded. To establish that $\mathcal G$ is
pregaussian it suffices, by Dudley's integral-criterion, to find a suitable
$\eta$-covering of $\mathcal G$ in the intrinsic covariance metric $d(s,t) :=
\|h_t-h_s\|_{2,P}$, for every $h_t,h_s \in \mathcal G$.

Consider first increments for $s<t, |s-t| \le 1, \min(|s|,|t|) \ge \zeta$,
\begin{eqnarray*}
h_t(x)-h_s(x) &=&{\cal F}^{-1}[(\phi^{-1})'(-u){\cal F}[g_t-g_s](u)](x)+{\cal F}^{-1}[\phi^{-1}(-u){\cal F}[g_t-g_s](u)](x)ix \\
&=&  i{\cal F}^{-1}[\phi^{-1}(-u) {\cal F}[x(g_t-g_s)](u)](x) \\
&=& i {\cal F}^{-1}[\phi^{-1}(-u)]\ast{\bf 1}_{(s,t]}(x),
\end{eqnarray*}
for which Proposition \ref{PropL2PEst} yields, with $f = 1_{(s,t]}$, the H\"older-type bound
\[ \norm{h_t-h_s}_{2,P}^2\lesssim \norm{\sin((t-s)u)u^{-1}(1+\abs{u})^{1-\eps}}_{L^{2+4/\eps}}^2+
\abs{t-s}\lesssim \abs{t-s}^{\eps(3+2\eps)/(2+\eps)}.
\]
This will give us a polynomially growing covering of $\mathcal G$ for all $t$ in a fixed compact interval.

To deal with large $|t|$ we shall establish the polynomial decay bound
$\norm{h_t}_{2,P}\lesssim |t|^{-1/2}$ as $\abs{t}\to\infty$, and we shall do
this for each of the three terms in the second line of (\ref{htrig})
separately.

For the first term, say $h_t^{(1)}$, this follows from $$\|h^{(1)}_t\|^2_{2,P}
\le\|h^{(1)}_t\|^2_{L^\infty} \le \|(\phi^{-1})'(-\cdot)\mathcal Fg_t\|^2_{L^1}
\le \|(\phi^{-1})'(-\cdot)\|^2_2 \|g_t\|^2_{L^2} \lesssim \int_{-\infty}^t
|x|^{-2} \sim |t|^{-1}$$ as $t \to -\infty$, and likewise for $t \to \infty$,
using the Cauchy-Schwarz inequality and Lemma \ref{LemmaLevyProp}(b).

For the second term $h_t^{(2)}$ we use the Cauchy-Schwarz inequality, the
finite second moment of $P$, Assumption \ref{gen}(c) and Lemma \ref{adm} to the
effect that $$\|h^{(2)}_t\|_{2,P}^2 \le \int x^2P(dx) \|\phi^{-1}(-u){\cal
F}g^c_t(u)\|^2_{L^1} \lesssim \|\mathcal F^{-1}[\phi^{-1}]\|^2_{H^{-1}}
\|g^c_t\|^2_{H^1}  \lesssim (1+|t|)^{-1}.$$

For the third term, since $xP$ has a bounded density by Lemma \ref{LemmaLevyProp}(a), it suffices to bound
\[ \norm{{\cal F}^{-1}[\phi^{-1}(-u){\cal F}g_t^s](x)|x|^{1/2}}_{L^2},\]
which by the Cauchy-Schwarz inequality can be estimated by
\[\norm{{\cal F}^{-1}[\phi^{-1}(-u){\cal F}g_t^s]x}_{L^2}^{1/2} \norm{{\cal F}^{-1}[\phi^{-1}(-u){\cal F}g_t^s]}_{L^2}^{1/2}.
\]
Now by Lemma \ref{adm} we know $\abs{{\cal F}g_t^s(u)}\lesssim (1+\abs{u})^{-1}(1+\abs{t})^{-1}, |{\cal F}[xg_t^s](u)|\lesssim (1+|u|)^{-1}$ and since $|(\phi^{-1})'|(u) \lesssim (1+|u|)^{-1}|\phi^{-1}(u)|$ from the proof of Lemma \ref{LemmaLevyProp} we can estimate the product in the last display to obtain the overall bound $$\|h_t^{(3)}\|_{2, P} \lesssim (1+|t|)^{-1/2} \|\phi^{-1}(-u)(1+|u|)^{-1}\|_{L^2} \lesssim (1+|t|)^{-1/2}$$ in view of Assumption \ref{gen}(c).

In conclusion, we can construct an $\eta$-covering of $\mathcal G$ by the functions $(i \Delta)^{-1}h_{t_i}$ with $t_i=i/M$ and $i=-M^2,\ldots,+M^2$ where $M=M(\eta)$ grows polynomially in $\eta^{-1}$. This shows that the covering numbers corresponding to this $\eta$-net satisfy
\begin{equation} \label{vcgauss}
\log(N(\mathcal G, L^2(P), \eta)) \lesssim \log(\eta^{-1}).
\end{equation}
The square-root of this entropy bound is integrable at zero as a function of $\eta$, which completes the proof by Dudley's continuity criterion (Theorem 2.6.1 in \cit{D99}).
\end{proof}

\subsubsection{Uniform CLT for the linear term}

\begin{theorem} \label{main}
Grant Assumption \ref{gen} and
\[(\nu_n^\varphi (t_1), \dots \nu_n^\varphi (t_k)) \to^\mathcal L (\mathbb G^\varphi (t_1), \dots, \mathbb G^\varphi (t_k))
\]
as $n \to \infty$ for every finite set $(t_1, \dots, t_k) \subset (-\zeta, \zeta)^c$. If $h_n\gtrsim n^{-1/(4\alpha)}$ for some $\alpha>(1-\eps)/2$, so in particular if $h_n \thicksim n^{-1/2}(\log n)^{-\rho}$ for some $\rho>1$, then
\begin{equation*}
\nu_n^\varphi \to^\mathcal L \mathbb G^\varphi ~~in ~~ \ell^\infty ((-\zeta, \zeta)^c)
\end{equation*}
as $n \to \infty$.
\end{theorem}

\begin{proof}
We set $\Delta=1$ and suppose that the kernel is symmetric, i.e. ${\cal F}K_h(-u)={\cal F}K_h(u)$, to ease notation. Given convergence of the finite-dimensional distributions it suffices to prove uniform tightness of $\{\nu_n^\varphi\}_{n \in \mathbb N}$ in $\ell^\infty((-\zeta, \zeta)^c)$, cf.~\cit{VW96}, Chapter 1.5. We shall in what follows decompose
$\nu_n^\phi$ into a sum of several processes indexed by $t$, and prove tightness of each of these processes separately, which implies tightness of the sum of the processes by the asymptotic equicontinuity characterisation of tightness in $\ell^\infty((-\zeta, \zeta)^c)$ (e.g., Theorem 1.5.7 in \cit{VW96}) and by the triangle inequality. We shall also frequently use the simple fact that tightness is preserved under isometric injections of $\ell^\infty((-\zeta,\zeta)^c)$: if $\nu$ is a process indexed by $s$ and $\nu'$ a process indexed by functions $f_s \in \mathcal F$, and if $\nu(s)=\nu'(f_s)$ for every $s \in(-\zeta,\zeta)^c$, then tightness of $\nu'$ in $\ell^\infty(\mathcal F)$ (normed by $\|H\|_\mathcal F := \sup_{f\in \mathcal F}|H(f)|$) implies tightness of $\nu$ in $\ell^\infty((-\zeta,\zeta)^c)$.

We decompose $g_t=g_t^c+g_t^s$ as in the proof of Lemma \ref{adm} with the particular choice $\chi(x)=e^x1_{(-\infty,0]}(x)$ for $t<0$, and symmetrically if $t>0$. The integrand of $\nu_n^\phi(t)$ in \eqref{EqNun} equals
\begin{align*}
 &{\cal F}^{-1}[\phi^{-1}(-u){\cal F}[ixg_t^s]{\cal F}K_h](x)+ {\cal F}^{-1}[\phi^{-1}(-u){\cal F}g_t^s{\cal F}[ixK_h]](x)\\
&\quad + {\cal F}^{-1}[\phi^{-1}(-u){\cal F}g_t^c{\cal F}K_h](x)ix+ {\cal F}^{-1}[(\phi^{-1})'(-u){\cal F}g_t^c{\cal F}K_h](x)\\
&=:(T_1+T_2+T_3+T_4)(x)
\end{align*}
The process indexed by the component $T_1$ is critical and its tightness is proved in Section \ref{SecCritical} below.

Concerning $T_2$, we have $\abs{\phi^{-1}(-u){\cal F}g_t^s{\cal
F}[ixK_h]}\lesssim\abs{\phi^{-1}(-u)} (1+\abs{u})^{-2}$ by
$\norm{xK_h}_{L^1}+\norm{(xK_h)'}_{L^1}\lesssim 1$, uniformly in $h$, and by
the admissibility of $g_t$. By Assumption \ref{gen}(c) we deduce that $T_2$
lies in a fixed norm ball of $H^1(\R)$. For $T_4$ we note $|(\phi^{-1})'|
\lesssim 1$, $\sup_{h>0,u}|\mathcal FK_h(u)| \le \|K\|_1<\infty$,
$\sup_{|t|\ge \zeta}\|g^c_t\|_{H^1}<\infty$ by Lemmas \ref{LemmaLevyProp} and
\ref{adm}, so $\{{\cal F}^{-1}[(\phi^{-1})'(-u){\cal F}g_t^c{\cal
F}K_h](\cdot), |t|\ge \zeta\}$ is bounded in $H^1(\R)$. For $T_3$ we use
$|\phi^{-1}(u)| \le (1+|u|)^{(1-\varepsilon)/2}$ and $$\|\mathcal F^{-1} \left[\phi^{-1}(-u) {\cal
F}g_t^c {\cal F}K_h\right]\|_{H^{(1+\varepsilon)/2}} \lesssim \|(1+|u|) {\cal
F}g_t^c\|_{L^2} = \|g_t^c\|_{H^1} <\infty,$$ uniformly in $|t|\ge \zeta$, again
by Lemmas \ref{LemmaLevyProp} and \ref{adm}. We conclude that the norms
$\|T_2+T_4\|_{H^1}$ and $\norm{T_3/x}_{H^{(1+\varepsilon)/2}}$ are bounded
uniformly in $t\in(-\zeta,\zeta)^c, h>0$. Each summand in $T_2+T_3+T_4$ is
therefore contained in a fixed $P$-Donsker-class: For $T_2+T_4$ this follows
from Proposition 1 in \cit{NP07} with $s=1, p=q=2$, and for $T_3$ we apply
Corollary 5 for weighted Besov-Sobolev spaces in \cit{NP07} with parameter
choice $s=(1+\varepsilon)/2$, $\beta=-1$, $p=q=2$, $\gamma=\eps/2$ noting that
the moment condition there is satisfied by
\eqref{EqMoments}. The empirical process $\nu_n^\phi$ is thus indexed by
functions $T_2+T_3+T_4$ that change with $n$ but that are contained in a
fixed $P$-Donsker class, and so is tight by the asymptotic equicontinuity
criterion.  Together with the tightness of the critical term, derived below,
this proves tightness of $\nu_n^\phi$.
\end{proof}

Combining the convergence of the finite-dimensional distributions from Corollary \ref{CorFidi} with Theorem \ref{main} and the uniform bounds on the remainder and bias term we have succeeded in proving Theorem \ref{Thm1}.

\subsubsection{The critical term}\label{SecCritical}

Note that in the ill-posed case $\lim_{\abs{u}\to\infty}|\phi(u)|=0$, for instance when $\phi(u)=(1-iu)^{-\alpha}$, the class involving $T_1$ with $\mathcal F K_h=1$ \textit{is not} $P$-Donsker even for $P$ with bounded density. The reason is, roughly speaking, that $\mathcal F^{-1}[\phi^{-1}(-\cdot)] \ast (e^{\cdot-t}1_{(-\infty, t]})$ is then unbounded at $t$, and classes that contain functions unbounded at any point cannot be Donsker for such $P$, cf.~the proof of Theorem 7 in \cit{N06}. This implies that one cannot use $h=0$, i.e., $K_h=\delta_0$, in the proofs, as could have been done in the 'noncritical' terms $T_2, T_3, T_4$ above. Rather, one needs to exploit the fact that the kernel $K_h$ smooths out the singularities for $h$ fixed, and if $h_n$ does not approach zero too fast, there is still hope to obtain a uniform central limit theorem, as shown in a different but conceptually related situation of Theorems 9 and 10 in \cit{GN08}.

As compactly supported kernels facilitate the arguments considerably, we introduce the truncated kernel
\[K_h^{(0)}:=K_h{\bf  1}_{[-\zeta/2,\zeta/2]}.\]
By the decay of $K$ and $K'$ from \eqref{EqKProp} we can again treat the term
involving $K_h-K_h^{(0)}$ by classical methods. Using
$\norm{K_h-K_h^{(0)}}_{BV} \lesssim h^{\beta-2}$ where $\|\cdot\|_{BV}$ is the
usual bounded variation norm, we obtain
\[ \abs{\phi^{-1}(-u){\cal F}[ixg_t^s]{\cal F}[K_h-K_h^{(0)}](u)}
\lesssim \abs{\phi^{-1}(-u)}(1+\abs{u})^{-2}h^{\beta-2},
\]
whence ${\cal F}^{-1}[\phi^{-1}(-u){\cal F}[ixg_t^s]{\cal F}[K_h-K_h^{(0)}]]\in
H^1(\R)$ follows, even with in $h$ shrinking and in $t$ uniform norms. As for
the terms $T_2, T_4$ above, we thus deduce the uniform tightness of this term
since norm balls in $H^1(\mathbb R)$ are universally Donsker.

Recalling $g_t^s(x)=x^{-1}e^{x-t}{\bf 1}_{(-\infty,t]}(x)$, the term involving the truncated kernel can be written as
\[ {\cal F}^{-1}[\phi^{-1}(-u){\cal F}[ixg_t^s]{\cal F}K_h^{(0)}]=iq(\cdot-t)\ast K_h^{(0)}
\]
with
\begin{equation} \label{q}
q(x):={\cal F}^{-1}[\phi^{-1}(-u)(1+iu)^{-1}](x).
\end{equation}
The regularity of $q$ in the scale of Besov spaces $B^s_{p,r}(\mathbb R)$ is $s=(1+\eps)/2$ for $p=1$ and $r=\infty$: Since $m(u)=\phi^{-1}(-u)(1+iu)^{-1/2+\eps/2}$ is  a Fourier multiplier on $B^{(1+\eps)/2}_{1,\infty}(\R)$ by Lemma \ref{LemmaLevyProp}(c), this assertion follows from the fact that $${\cal F}^{-1}[(1+iu)^{-1/2-\eps/2}](x)=\Gamma(1/2+\eps/2)^{-1}\abs{x}^{\eps/2-1/2}e^x{\bf 1}_{(-\infty,0]}(x)$$
(a Gamma-type density) is an element of that space. The latter follows either by checking directly that its $L^1$-modulus of smoothness satisfies $\omega(h)_1\lesssim h^{1/2+\eps}$ or by noting that multiplication by $(1+iu)^{(1-\eps)/2}$ in the Fourier domain is an isomorphism between $B^{(1+\eps)/2}_{1,\infty}(\R)$ and $B^1_{1,\infty}(\R)$ and ${\cal F}^{-1}[(1+iu)^{-1}](x)= e^x{\bf 1}_{(-\infty,0]}(x)$ is of bounded variation and thus contained in $B^1_{1,\infty}(\R)$. Moreover, by embedding theorems for Besov spaces, $q$ is then also an element of $B^s_{1,1}(\mathbb R)$ for any $s<(1+\eps)/2$ and thus also of $L^1(\mathbb R) \cap L^2(\mathbb R)$ . We refer to \cit{Tr10} for these standard properties of Besov spaces.

We are thus left with proving tightness of
\begin{equation} \label{sep}
\sqrt n \int_{\R} \big(q(\cdot-t) \ast K_h^{(0)}\big)(x) (P_n-P)(dx) = \sqrt n \int_\mathbb R q(y-t)
\big(K_h^{(0)} \ast (P_n-P)\big)(y)\,dy, \quad \abs{t}\ge\zeta,
\end{equation}
which is a smoothed empirical process indexed by  \begin{equation}\label{EqF}
\mathcal F = \left\{ q(\cdot-t) : |t| \ge \zeta \right\}.
\end{equation} The following general purpose result follows from the proof of Theorem 3 in \cit{GN08}, which builds on fundamental ideas in the classical paper \cit{GZ84}, and can be applied to the unbounded processes relevant here. For a given class of measurable functions $\mathcal{F}$ we write
\begin{equation*}
\mathcal{F}_{\delta }^{\prime }=\{f-g:f,g\in \mathcal{F},\left\Vert
f-g\right\Vert _{2,P}\le \delta \}.
\end{equation*}
We shall rather loosely use the standard empirical process terminology from \cit{GN08}.

\begin{theorem}
\label{pregaussian}Let $\mathcal{F}$ be any $P$-pregaussian class of
real-valued functions on $\mathbb{R}^{d}$ and let $\{\mu _{n}\}_{n=1}^{\infty }$ be
a sequence of finite signed measures defined on $\mathbb R^d$ satisfying $\sup_n \|\mu_n\|<\infty$. Let $\bar {\mu}_n(A)=\mu_n(-A)$. Assume that $\mathcal{F}\subseteq L^{1}(|\mu _{n}|)$ holds for every $n$ and, in addition,
\begin{enumerate}
\item for each $n$, the class $\tilde{\mathcal{F}}_{n}:=\{f\ast
\bar{\mu}_{n}:f\in \mathcal{F}\}$ consists of functions whose absolute
values are bounded by a constant $M_{n}$;

\item $\sup_{f\in \mathcal{F'_\delta}}E(f\ast \bar{\mu}_{n}(X))^{2} \le 4 \delta^2$ for every $\delta>0$ and $n \ge n_0\equiv n_0(\delta)$ large enough;

\item for i.i.d.~Rademacher variables $(\eps_i)_i$, independent of the $X_i$'s, we have
\begin{equation}
\left\Vert \frac{1}{\sqrt{n}}\sum_{i=1}^{n}\varepsilon
_{i}f(X_{i})\right\Vert _{(\tilde{\mathcal{F}}_{n}
)'_{1/n^{1/4}}}\rightarrow 0  \label{resdons}
\end{equation}
as $n\rightarrow \infty $ in outer probability;

\item $\cup_{n \ge 1}\tilde{\mathcal F}_n$ is in the $L^2(P)$-closure of $\sup_n\|\mu_n\|$-times the symmetric convex hull of some fixed $P$-pregaussian class of functions $\bar {\mathcal F}$.

\item For all $0<\eta <1$, the $L^2(P)$-metric entropy of $\tilde {\mathcal F}_n$ satisfies $H(\tilde{\mathcal{F}}_{n},L^{2}(P),\eta )\le \lambda _{n}(\eta )/\eta ^{2}$ for functions $\lambda _{n}(\eta )$ such that $\lambda _{n}(\eta
)\rightarrow 0$ and $\lambda _{n}(\eta )/\eta ^{2}\rightarrow
\infty $ as $\eta \rightarrow 0$, uniformly in $n$, and the
bounds $M_{n}$ of part (a) satisfy
\begin{equation}
M_{n}\le \left( 5\sqrt{\lambda _{n}(1/n^{1/4})}\right) ^{-1}
\label{envelope}
\end{equation}%
for all $n$ large enough.
\end{enumerate}
Then $\sqrt{n}(P_{n}-P)\ast \mu _{n}$ is uniformly tight in the Banach space $\ell^\infty (\mathcal F)$ (equipped with the uniform norm $\|\cdot\|_\mathcal F$).
\end{theorem}
\begin{proof}
The differences to Theorem 3 in \cit{GN08} are: We do not require $\mu_n(\mathbb R)=1~\forall n$, and (b) is slightly weakened, both permitted as we only establish tightness in this theorem and not convergence of the finite-dimensional distributions. Moreover the new condition (d), which replaces translation invariance of $\mathcal F$ by a more generic condition. Note that Theorem 0.3 in \cit{D73} implies that $L^2(P)$-closures of symmetric convex hulls of pregaussian classes are again pregaussian, which is all that is needed for the proof of Theorem 3 in \cit{GN08} to apply.
\end{proof}

We now verify these conditions for the classes above, with $d\mu_n(y) =K^{(0)}_h(y)dy$. Let us first show that the class $\mathcal F$ from (\ref{EqF}) is indeed $P$-pregaussian. By Proposition \ref{PropL2PEst} applied to
\[f(x)=e^x(e^{-t}{\bf 1}(x\le t)-e^{-s}{\bf 1}(x\le s)),\quad t,s\le-\zeta\text{  (and symmetrically for $t,s\ge\zeta$)}
\]
and by  the same estimates as in the proof of Theorem \ref{scgauss}
\begin{equation}\label{EqModulusq}
\norm{q(\cdot-t)-q(\cdot-s)}_{2,P}\lesssim \abs{t-s}^{\eps(3+2\eps)/(2+2\eps)}.
\end{equation}
Moreover,  the tail bound for the third term in that proof applies exactly here
such that the same arguments show that $\mathcal F$ has polynomially growing
covering numbers and is thus pregaussian. In particular, $\cal F$ is bounded in
$L^2(P)$. The functions $q(\cdot-t)$ are in $B^{(1+\eps)/2}_{1,\infty}(\mathbb
R)\subset L^1(\R) \cap L^2(\R)$ and thus in $L^1(|\mu_n|)$ since $K$ is
bounded.

\begin{enumerate}
 \item the envelopes of $q(\cdot-t) \ast K^{(0)}_h$ are of order
 $M_n \lesssim h^{-\alpha'}$ for $\alpha'\in((1-\eps)/2,\alpha)$ when
 $h=h_n\gtrsim n^{-1/(4\alpha)}$ since the sup-norm is bounded by the
 BV-norm, which in turn is bounded in point (c) below.

\item Let $g \in \mathcal F_\delta'$, then $\|K^{(0)}_h \ast g\|_{2,P} \le \|K^{(0)}_h \ast g - g\|_{2,P} + \delta$ and the result follows from the triangle inequality if we show $\|K^{(0)}_h \ast f - f\|_{2,P}\to 0$ uniformly over $f \in \mathcal F$. From \eqref{EqPhiL2} above, noting $\supp(K^{(0)}_h \ast (i\cdot g^s_t))\cap (-\zeta/2,\zeta/2)=\varnothing$, we conclude
\[ \|K^{(0)}_h \ast f - f\|_{2,P}\lesssim \norm{(1+\abs{u})^{-\eps}({\cal F}K_h^{(0)}-1)}_{L^{2+4/\eps}}+\norm{K_h^{(0)}\ast f-f}_{L^2}.
\]
Since ${\cal F}K_h^{(0)}(u)$ is uniformly bounded and tends to $1$ pointwise and since $(1+\abs{u})^{-(2\eps+4)}$ is integrable, by dominated convergence the first norm tends to zero for $h\to 0$. Similarly, as ${\cal F}f\in L^2$ holds, $\norm{({\cal F}K_h^{(0)}-1){\cal F}f}_{L^2}\to 0$ follows and by Plancherel's theorem also the second norm converges to zero. This convergence is uniform because of $\abs{{\cal F}f(u)}=\abs{{\cal F}q(u)}$ for all $f\in\mathcal F$ and since $q\in L^2(\R)$.

\item The class $\{K^{(0)}_h \ast q(\cdot-t): |t| \ge \zeta \}$ consist of translates of the fixed
function $K^{(0)}_h \ast q$, which is a function of bounded variation with
BV-norm of size $h^{-\alpha'}$ for some $\alpha'\in((1-\eps)/2,\alpha)$
using $q \in B^{1-\alpha'}_{11}(\mathbb R)$ from the argument after
(\ref{q}) and the estimate (61) in \cit{GN08} (whose proof applies also to
the truncated kernels). The envelope $M_n$ of $\tilde{\mathcal F_n}$ is
then of the same size since the BV-norm bounds the supremum norm. Moreover
the class $\{K^{(0)}_h \ast q(\cdot-t): |t|\ge \zeta \}$ has polynomial
$L^2(Q)$-covering numbers, uniformly in all probability measures $Q$. To
see this we argue as in Lemma 1 in \cit{GN09}: note that a function of
bounded variation is the composition of a $1$-Lipschitz function with a
monotone function. The set of all translates of a monotone function has
VC-index 2, and hence has polynomial covering numbers by Theorem 5.1.15 in
\cit{PG99}, with constants $A,v$ there independent of $n$. Composition with a
$1$-Lipschitz map preserves the entropy, and the estimate (22) in
\cit{GN08} with envelopes $M_n \sim h_n^{-\alpha'}$ and $H_n(\eta) \equiv
H(\eta)\thicksim\log (\eta)$ now shows that
\begin{equation*}
E\left\Vert \frac{1}{\sqrt{n}}\sum_{i=1}^{n}\varepsilon_{i}f(X_{i})\right\Vert _{(\tilde{\mathcal{F}}_{n})'_{1/n^{1/4}}} \lesssim \max \left[\frac{\sqrt {\log n}}{n^{1/4}}, \frac{h_n^{-\alpha'}}{\sqrt n} \log n\right] \to 0
\end{equation*}
as $n \to \infty$, in view of $h_n^{-\alpha'}\lesssim h_n^{-\alpha}\lesssim n^{1/4}$.

 \item
Using that $K_h^{(0)}$ is supported in $[-\zeta/2, \zeta/2]$, one shows by standard arguments that the class of functions
$$\bigcup_{h>0} \left\{x \mapsto \int_\mathbb R q(x-t-y)K^{(0)}_h(y)dy: |t| \ge \zeta \right\}$$
is in the $L^2(P)$-closure of $\|K\|_{L^1}$-times the symmetric convex hull of the $P$-pregaussian class $\bar {\mathcal F}=\left\{q(\cdot-t): |t| \ge \zeta/2\right\}.$ To see this one can either make a minor modification of the argument in Lemma 1 in \cit{GN08}, or notice that, $\left\{q(\cdot-t): |t| \ge \zeta/2\right\}$ being bounded in the separable Banach space $L^2(P)$ (cf.~after (\ref{EqModulusq})), the integrals $\int q(\cdot-t-y)K^{(0)}_h(y)dy$ are $L^2(P)$-valued Bochner-integrals, and can thus be obtained as $L^2(P)$-limits of simple functions lying in the symmetric convex hull of $\{z \mapsto \|K\|_{L^1} q(z-t): |t| \ge \zeta/2\}$ (e.g., Appendix E and Theorem E.3 in \cit{D99}).

\item Write $f, g$ for distinct translates of $q$ (elements of $\mathcal F$), and deduce from Minkowski's inequality for integrals that
\begin{eqnarray*}
\left(E\left[(f \ast K_h^{(0)}(X) - g \ast K_h^{(0)}(X))^2\right]\right)^{1/2} &\le& \int_{-\zeta/2}^{\zeta/2} |K_h(u)| \|f(-u-\cdot)-g(-u-\cdot)\|_{2,P}du \\
&\le& \|K\|_{L^1} \sup_{|u| \le \zeta/2}\|f(u-\cdot)-g(u-\cdot)\|_{2,P} .
\end{eqnarray*}
 Since entropy bounds are preserved under Lipschitz transformations, and since $$\{q(u-\cdot-t): |t| \ge \zeta, |u| \le \zeta/2\} \subset \{q(u-\cdot-t): |t| \ge \zeta/2\}$$ has polynomial $L^2(P)$-covering numbers by the same arguments as after (\ref{EqModulusq}), we deduce the bound $H(\tilde {\mathcal F}_n, L^2(P), \eta)\lesssim \log (\eta^{-1})$ for  every $\eta>0$ small enough, independent of $n$. Conclude that we can take $\lambda_n (\eta) = \log (\eta^{-1}) \eta^2$, so that the envelope condition \eqref{envelope} becomes
\begin{equation} \label{envcon}
h_n^{-\alpha'} \lesssim (\log n)^{-1/2} n^{1/4},
\end{equation}
which is satisfied due to $\alpha'<\alpha$ and $h_n^{-\alpha}\lesssim n^{1/4}$, completing the proof.
\end{enumerate}

\medskip

\noindent \textbf{Acknowledgement.} The authors would like to thank an anonymous referee, Jakob S\"{o}hl and Mathias Trabs for a careful reading of the manuscript and several helpful suggestions and corrections, as well as Chris Klaassen for useful discussions.

\end{document}